\documentclass[11pt]{amsart}
\usepackage{amsfonts,amssymb,amsthm}
\usepackage{amsmath,amscd}
\usepackage{pstricks}
\usepackage{pstricks,pst-node}
\usepackage{mathrsfs}
\usepackage[all]{xy}
\usepackage{enumitem}
\DeclareMathAlphabet{\mathpzc}{OT1}{pzc}{m}{it}

\theoremstyle{plain}
\newtheorem{theorem}{Theorem}[section]
\newtheorem{proposition}[theorem]{Proposition}

\newtheorem{lemma}[theorem]{Lemma}

\theoremstyle{definition}

\newtheorem{remark}[theorem]{Remark}

\numberwithin{equation}{section}

\newcommand{\id}{\operatorname{id}}
\newcommand{\End}{\operatorname{End}}
\newcommand{\Hom}{\operatorname{Hom}}
\newcommand{\Coker}{\operatorname{Coker}}

\newcommand{\Spec}{\operatorname{Spec}}

\newcommand{\Frac}{\operatorname{Frac}}
\newcommand{\add}{\operatorname{add}}

\newcommand{\wo}{w_{0,r}}
\newcommand{\X}{x_{(r,0,\ldots,0)}}

\def\oop{{\text{\rm op}}}

\newcommand{\mpk}{\mathpzc K}
  \newcommand{\vep}{\varepsilon}

\newcommand{\bfi}{{\mathbf{i}}}

\newcommand{\bfU}{{\mathbf{U}}}

\def\fS{{\frak S}}

\newcommand{\msD}{\mathscr D}

\newcommand{\msH}{\mathsf{H}}
\newcommand{\msT}{\mathsf{T}}
\newcommand{\msS}{\mathsf{S}}
\newcommand{\msP}{\mathsf{P}}
\newcommand{\msZ}{\mathsf{Z}}
\newcommand{\msC}{\mathsf{C}}

\def\sfz{{\mathsf z}}

\def\sH{{\mathcal H}}

\def\sS{{\mathcal S}}

\def\sZ{{\mathcal Z}}

\newcommand{\mbn}{\mathbb N}
\newcommand{\mbq}{\mathbb Q}

\newcommand{\mbz}{\mathbb Z}

\newcommand{\ttf}{\mathsf{f}}
\newcommand{\ttg}{\mathsf{g}}

\newcommand{\tth}{\mathtt{h}}

\newcommand{\spann}{\operatorname{span}}

\newcommand{\la}{{\lambda}}
\newcommand{\La}{\Lambda}
\newcommand{\ga}{{\gamma}}

\newcommand{\Og}{\Omega_\sZ}
\newcommand{\Ogk}{\Omega_{\mpk}}

\newcommand{\og}{\omega}

\newcommand{\up}{\upsilon}
\newcommand{\vi}{\varphi}

\newcommand{\sg}{\sigma}

\def\ggp#1#2{\left[\kern-3.2pt\left[{#1\atop #2}\right]\kern-3.2pt\right]}

\def\leq{\leqslant}\def\geq{\geqslant}
\def\le{\leqslant}\def\ge{\geqslant}

\newcommand{\bop}{\bigoplus}

\newcommand{\ot}{\otimes}

\newcommand{\han}{\subseteq}

\newcommand{\lra}{\longrightarrow}
\newcommand{\ra}{\rightarrow}

\newcommand{\zr}{\zeta_r}
\newcommand{\xr}{\xi_r}

\newcommand{\afmsD}{{\mathscr D}^\vtg}

\newcommand{\mnmod}{\!\!\!\mod\!}

\newcommand{\vtg}{{\!\vartriangle\!}}

\newcommand{\afSrk}{{\mathcal S}_{\vtg}(n,r)_{\mathpzc K}}
\newcommand{\afHrk}{{\mathcal H}_{\vtg}(r)_{\mathpzc K}}
\newcommand{\Hrk}{{\mathcal H}(r)_{\mathpzc K}}

\newcommand{\fSr}{\fS_r}
\newcommand{\affSr}{{\fS_{\vtg,r}}}

\newcommand{\afHr}{{\sH_\vtg(r)}_\sZ}

\newcommand{\afbfHr}{{\boldsymbol{\mathcal H}_\vtg(r)}}

\newcommand{\afSr}{{\mathcal S}_{\vtg}(n,r)_\sZ}

\newcommand{\afbfSr}{{\boldsymbol{\mathcal S}}_\vtg(n,r)}

\newcommand{\afmbnn}{\mathbb N_\vtg^{n}}

\newcommand{\afLanr}{\Lambda_\vtg(n,r)}

\begin{document}
\title{Affine quantum Schur--Weyl duality}

\author{Qiang Fu}
\address{School of Mathematical Sciences,
Key Laboratory of Intelligent Computing and Applications (Ministry of Education),
Tongji University, Shanghai, 200092, China.}
\email{q.fu@hotmail.com, q.fu@tongji.edu.cn}
\author{Jun Hu}
\address{Key Laboratory of Algebraic Lie Theory and Analysis of Ministry of Education,
School of Mathematics and Statistics, Beijing Institute of Technology, Beijing, 100081, P.R. China}
\email{junhu404@bit.edu.cn}
\thanks{Supported by the National Natural Science Foundation
of China (12371032, 12431002)}

\begin{abstract}
Let $\mpk$ be an arbitrary commutative ring containing an invertible element
$\vep$. Let $\afHrk$ be the extended affine Hecke algebra of type $A$ with Hecke parameter $\vep$, let $\Ogk^{\otimes r}$ be the
affine tensor space, and let $\afSrk$ be the corresponding affine
quantum Schur algebra. We first prove that the natural right action
of $\afHrk$ on $\Ogk^{\otimes r}$ is always faithful.

Assume further that $\mpk$ is a field of
characteristic $0$ and that $\vep$ is not a root of unity. We prove
that, for any $n\geq 2$, the natural algebra homomorphism
$
\xi_r:\afHrk\rightarrow
\End_{\afSrk}(\Ogk^{\otimes r})^{\mathrm{op}}$
is an isomorphism. This proves Conjecture~3.8.8 of \cite{DDF}.
As an application, we prove the conjecture formulated in \cite[5.2.4]{DDF} concerning the center of the affine quantum Schur algebra.

We also prove that $\afSrk$ is left and right Noetherian whenever
$\mpk$ is a Noetherian commutative ring, which verify a conjecture in \cite[Rem. 1.7]{DY}.
\end{abstract}

 \sloppy \maketitle
\section{Introduction}

Schur-Weyl duality stands as one of the most foundational and influential correspondences in classical representation theory, establishing a canonical commuting action duality between the symmetric group $\mathfrak{S}_r$ and the general linear group $\mathrm{GL}_n(\mathbb{C})$ on the tensor space $(\mathbb{C}^n)^{\otimes r}$. This classical framework (\cite{Schur}, \cite{Weyl}) elegantly interconnects the finite-dimensional representation theories of symmetric groups and general linear groups over the complex field $\mathbb{C}$. Later pioneering works by Carter-Lusztig \cite{CL}, De Concini-Procesi \cite{CPr} and Thrall \cite{Th} (see also \cite{BD}) systematically established the positive-characteristic version of Schur-Weyl duality, which bridge the modular representation theories of symmetric groups and general linear groups through a double centralizer correspondence.

The type $A$ quantum Schur-Weyl duality, first introduced by Jimbo \cite{Jim}, replaces the general linear group $\mathrm{GL}_n(\mathbb{C})$ with the quantum enveloping algebra $\bfU(\mathfrak{gl}_n)$ over $\mathbb{Q}(v)$ and substitutes the symmetric group algebra $\mathbb{C}\mathfrak{S}_r$ with the Iwahori-Hecke algebra $\mathcal{H}(\mathfrak{S}_r)$ defined over the same field $\mathbb{Q}(v)$. In the early 1990s, the celebrated BLM construction due to Beilinson, Lusztig, and MacPherson \cite{BLM} gave a geometric realization of quantum $\frak{gl}_n$ by using a geometric construction of  quantum Schur algebras via partial flag varieties. Building on the core results of \cite{BLM}, Du \cite{Du95} established the surjection from the Lusztig integral form ${U}(\mathfrak{gl}_n)_{\sZ}$ to the integral quantum Schur algebra, where $\sZ=\mbz[v,v^{-1}]$. This establishes one fundamental component of type $A$ quantum Schur-Weyl duality in the integral setting and, consequently, allows specialization to any root-of-unity parameters. The complementary component of the duality, namely the surjective homomorphism from the integral Iwahori-Hecke algebra $\mathcal{H}(\mathfrak{S}_r)_\sZ$ to the endomorphism algebra of the $v$-tensor space (viewed as a module over the integral quantum Schur algebra), was proved by Du, Parshall, and Scott in \cite{DPS}.

The subsequent developments were directed toward establishing an affine
analogue of quantum Schur--Weyl duality. Ginzburg--Vasserot \cite{GV}
and Lusztig \cite{Lu99} constructed algebra homomorphisms from quantum
affine $\frak{sl}_n$ to affine quantum Schur algebras. Deng, Du, and Fu \cite{DDF}
established the surjection from quantum affine $\mathfrak{gl}_n$ onto
affine quantum Schur algebras over $\mbq(v)$. In the classical affine
setting, Fu \cite{Fu13} established an integral analogue, namely a
surjection from an integral form of the universal enveloping algebra
of $\widehat{\mathfrak{gl}}_n$ onto the affine Schur algebra over
$\mathbb Z$. Subsequently, using the BLM realization of
quantum affine $\mathfrak{gl}_n$ developed in \cite{DF15}, Du--Fu \cite{DF19}
established the
corresponding integral affine quantum version, namely the
surjection from an integral form of quantum affine $\mathfrak{gl}_n$
onto the affine quantum Schur algebra over $\sZ$.
When $n\geq r$, the affine tensor space  $\Omega_{\mpk}^{\otimes r}$ contains the affine Hecke algebra $\afHrk$ as a direct summand, and it follows readily that the natural algebra homomorphism $\xi_r$ from $\afHrk$ to the endomorphism algebra of the affine tensor space $\Omega_{\mpk}^{\otimes r}$ (viewed as a module over the affine quantum Schur algebra) is an isomorphism. Deng, Du, and Fu conjectured in \cite[3.8.8]{DDF} that $\xi_r$ remains surjective when $n<r$, $\mpk=\mathbb{C}$ and $\vep$ is not a root of unity. One of the main results of this paper Theorem \ref{main} provides an affirmative answer
to this conjecture\footnote{Note that  Pouchin \cite{Pouchin}  developed a geometric approach to affine quantum Schur--Weyl duality. However a gap in the proof of the surjectivity statement in \cite[Th. 8.1]{Pouchin} was pointed out in \cite{Fu13}.}. As a further application, we prove that the center of the affine quantum Schur algebra ${\mathcal S}_{\vtg}(n,r)_{\mathbb{Q}(v)}$ is the polynomial algebra $\mathbb{Q}(v)[\sigma_1,\cdots,\sigma_{r-1},\sigma_r,\sigma_r^{-1}]$, where each $\sigma_i$ is defined as in \cite[\S5.2]{DDF}. This verifies the conjecture formulated in \cite[5.2.4]{DDF}.

Deng and Yang \cite{DY} showed that the affine quantum Schur algebra $\afbfSr$ over the rational functional field $\mathbb{Q}(v)$ is  Noetherian and affine quasi-hereditary in the sense of
\cite{Kl}. They conjectured in \cite[Rem. 1.7]{DY} that the affine quantum Schur algebra ${\mathcal S}_{\vtg}(n,r)_{\sZ}$ is Noetherian. In this paper, we prove this conjecture by establishing the
more general fact that the affine quantum Schur algebra ${\mathcal S}_{\vtg}(n,r)_{\mpk}$ is Noetherian for any Noetherian domain $\mpk$. In particular, this implies that ${\mathcal S}_{\vtg}(n,r)_{\sZ}$ is an affine quasi-hereditary $\sZ$-algebra in the sense of \cite{Kl} by \cite[Rem. 1.7]{DY}.

We now outline the main ideas and strategy of this paper. There are five highlights concerning our main results and approach. First, we show in Theorem \ref{injection} that for any $n,r$, any commutative ring $\mpk$ and any parameter $q\in\mpk^\times$, the right action of the extended affine Hecke algebra $\afHrk$ on the affine tensor space $\Omega_{\mpk}^{\otimes r}$ is faithful. In other words, the natural algebra homomorphism $\xi_r: \afHrk \rightarrow \operatorname{End}_{\afSrk} \left( \Omega_{\mpk}^{\otimes r} \right)^{\operatorname{op}}$ is always injective. This is already quite different with the classical quantum Schur-Weyl duality. In the classical theory, the natural algebra homomorphism $\mpk\mathfrak{S}_r\rightarrow\End((\mpk^n)^{\otimes r})^{\operatorname{op}}$ is injective if and only if $n\geq r$. When $n<r$, the description of the kernel of this natural homomorphism is addressed as the second fundamental theorem in invariant theory (\cite{CPr}, \cite{LZ}, \cite{BEG}). So this faithful result is a remarkable new phenomenon in the affine quantum Schur-Weyl duality theory.

Second, since all the modules involved in the affine quantum Schur-Weyl duality are infinite-dimensional, the semisimple representation theory method for classical quantum Schur-Weyl duality has no chance to be generalized to this affine quantum setting. To overcome this difficulty, we generalize Auslander-Solberg's ring theoretic approach (\cite{AS}) to double centralizer property to an infinite dimensional modules setting, see Lemma \ref{left approx} and Proposition \ref{embedding criterion}. Auslander-Solberg's result \cite{AS} considered only Artin algebras. We replace this Artin algebra assumption with a Noetherian assumption. We show in Proposition \ref{finite generated} that for any $n,r$, any Noetherian commutative ring $\mpk$ and any Hecke parameter $\vep\in\mpk^\times$, the affine quantum Schur algebra ${\mathcal S}_{\vtg}(n,r)_\mpk$ is Noetherian. Our modification of Auslander-Solberg's result \cite{AS} provide a general framework to deal with those Schur-Weyl reciprocity which may involves infinite-dimensional modules.

Third, we apply the general framework established by Lemma \ref{left approx} and Proposition \ref{embedding criterion} to study the the affine quantum Schur-Weyl duality when $\mpk$ is a field of characteristic $0$ and $\vep\in\mpk^\times$ is not a root of unity. Here the key point is to find an embedding of the cokernel ${\rm{Coker}}\,\ttf$ of the embedding $\ttf: \afHrk\rightarrow (\Omega_{\mpk}^{\otimes r})^{\oplus a}$ into some direct sum of copies of the affine tensor space $\Omega_{\mpk}^{\otimes r}$. We use the $Z(\afHrk)$ torsion-freeness of the cokernel ${\rm{Coker}}\,\ttf$ to construct the desired embedding in Lemma \ref{C embed H}, where $Z(\afHrk)$ denotes the center of the extended affine Hecke algebra $\afHrk$. This part of construction is quite general and of independent interest which may have further application in the study of double centralizer property in other context.

Forth, in order to prove the $Z(\afHrk)$ torsion-freeness of the cokernel ${\rm{Coker}}\,\ttf$, we consider the localization of the map $\ttf$ at the prime ideals of $Z(\afHrk)$ of height one. One of the crucial observation we used here is that the $Z(\afHrk)$ torsion-freeness of the cokernel ${\rm{Coker}}\,\ttf$ follows from the $Z(\afHrk)_{\mathfrak{p}}$ torsion-freeness of the cokernel ${\rm{Coker}}\,\ttf_{\mathfrak{p}}$ of the embedding $\ttf_{\mathfrak{p}}$ for every height one prime ideal $\mathfrak{p}$ of $Z(\afHrk)$, see Proposition \ref{height one torsion free}. Our proof makes use of Brown-Gordon-Stroppel's result \cite{BGS} which asserts the extended affine Hecke algebra $\afHrk$ of type $A$ is a free Frobenius extension of its center $Z(\afHrk)$. This observation is very useful and can be applied in the future study of the affine quantum Schur-Weyl duality at the integral setting and the roots of unity specializations.

Fifth, we analyze principal-series modules of the extended affine Hecke algebras $\afHrk$. Using the irreducibility criterion of Rogawski \cite{R}, as stated
in Chari--Pressley \cite[Prop.~3.4(c)]{CP}, we show that the localization $\ttf_{\mathfrak{p}}$ of the map $\ttf$ actually split for every height one prime ideal $\mathfrak{p}$ of $Z(\afHrk)$, see Lemma \ref{height one splitting}. The key point is to show for every height one prime ideal $\mathfrak{p}$ of $Z(\afHrk)$, the two-sided ideal of $\msH_{\mathfrak p}$ (i.e., the localization of $\afHrk$ at $\mathfrak{p}$) generated by a certain primitive idempotent $e$ of the semisimple Iwahori-Hecke algebra $\mathcal{H}(\mathfrak{S}_r)_\mpk$ coincides with $\msH_{\mathfrak p}$, see Proposition \ref{prop:height-one-fullness}. As a consequence, we show in Theorem \ref{main} that the natural algebra homomorphism $ \xr: \afHrk \longrightarrow \operatorname{End}_{\afSrk} (\Omega_{\mpk}^{\otimes r})^{\operatorname{op}}$ is always an isomorphism for all $n\geq2$.  That is, $$ \afHrk \xrightarrow{\sim} \operatorname{End}_{\afSrk} (\Omega_{\mpk}^{\otimes r})^{\operatorname{op}} . $$
Combining this result with the surjectivity theorem of Deng--Du--Fu from the quantum affine $\mathfrak{gl}_n$ to the affine quantum Schur algebra \cite[Th. 3.8.1]{DDF}, we obtain the complementary centralizer statement for the quantum affine $\mathfrak{gl}_n$-action on the affine tensor space. Thus the present work completes the affine Hecke algebra side of affine quantum Schur--Weyl duality.
As a further application, we determine in Theorem \ref{center Z0} the center of the affine quantum Schur algebra and prove a conjecture formulated in \cite[5.2.4]{DDF}.

The paper is organized as follows. In \S 2, we recall the extended affine Hecke algebra, affine quantum Schur algebra and affine tensor space, and define the natural algebra homomorphism $ \xr: \afHrk \longrightarrow \operatorname{End}_{\afSrk} (\Omega_{\mpk}^{\otimes r})^{\operatorname{op}}. $ In \S 3 we prove that $\xi_r$ is injective. In \S 4, we generalize Auslander-Solberg's ring theoretic approach (\cite{AS}) to double centralizer property to an infinite dimensional modules setting, and establish an abstract criterion for double centralizer property which may involve infinite dimensional modules. In \S 5 we first show in Proposition \ref{finite generated} that for any $n,r$, any Noetherian commutative ring $\mpk$ and any Hecke parameter $\vep\in\mpk^\times$, the affine quantum Schur algebra ${\mathcal S}_{\vtg}(n,r)_\mpk$ is Noetherian.
In particular,
by \cite[Rem.~1.7]{DY},
${\mathcal S}_{\vtg}(n,r)_\sZ$ is affine quasi-hereditary
in the sense of \cite{Kl}.  Then we study the height-one structure of the affine Hecke algebra through principal-series representations. In \S6 we prove in Theorem \ref{main} that $\xi_r$ is an algebra isomorphism when $\mpk$ is a field of characteristic $0$ and $\vep$ is not a root of unity. In \S 7 we use the main result Theorem \ref{main} to determine in Theorem \ref{center Z0} the center of the affine quantum Schur algebra and proves a conjecture formulated in \cite[5.2.4]{DDF}.

\section{Affine Schur--Weyl setting}

In this section we recall the basic objects involved in affine quantum Schur--Weyl duality. We first introduce the extended affine Hecke algebra of type $A$ and then define the affine tensor space and the associated affine quantum Schur algebra. The commuting actions on the tensor space give rise to the natural algebra homomorphism $\xi_r$, whose injectivity will be proved in the next section.

\subsection{Extended affine Hecke algebras of type $A$}
We begin with the extended affine symmetric group and the associated affine Hecke algebra.

Let $\affSr$ be the (extended) {affine symmetric group} consisting of all permutations
$w:\mbz\ra\mbz$ satisfying $w(i+r)=w(i)+r$ for $i\in\mbz$.
Let $W_r$ be the subgroup of
$\affSr$, the {\it Weyl group} of affine type $A$, generated by $S=\{s_i\}_{1\leq i\leq r}$, where $s_i$  is defined by
$s_i(j)=j$ for $j\not\equiv i,i+1\mnmod r$, $s_i(j)=j-1$ for
$j\equiv i+1\mnmod r$, and $s_i(j)=j+1$ for $j\equiv i\mnmod r$.
Let $\rho$ be the permutation of $\mbz$ sending $j$ to $j+1$ for all $j\in\mbz$. We extend the length function $\ell$ on $W_r$ to $\affSr$ by setting $\ell(\rho^mw)=\ell(w)$ for all $m\in\mbz,w\in W_r$.

Let $\sZ=\mbz[\up,\up^{-1}]$, where $\up$ is an indeterminate over $\mathbb{Z}$. The extended affine Hecke algebra $\afHr$ over $\sZ$ associated to
$\affSr$ is the $\sZ$-algebra which is free as $\sZ$-module with basis $\{T_w\}_{w\in\affSr}$, and is generated by $T_\rho,T_{\rho^{-1}},T_s, s\in S$ and whose multiplication rules are given by the formulas, for all $s\in S$ and $w\in\fS_{\vtg,r}$,
\begin{equation*}
\aligned
T_sT_w&=\begin{cases} (\up^2-1)T_{w}+\up^2T_{sw},\quad&\text{ if }\ell(sw)<\ell(w);\\
                                         T_{sw},\quad&\text{ if } \ell(sw)=\ell(w)+1,\end{cases}\\
T_\rho T_w&=T_{\rho w}.\\
\endaligned
\end{equation*}
Let $\afbfHr=\afHr\ot_\sZ\mbq(v)$, where $\mbq(\up)$ be the fraction field of $\sZ$.

The affine Hecke algebra $\afHr$ also admits the following {Bernstein
presentation}  which consists of
generators
$$T_i,\quad X_j\;(\text{$1\leq i\leq r-1$, $1\leq j\leq r$}),$$
 and relations
$$\aligned
 & (T_i+1)(T_i-\up^2)=0,\,\,\,\,T_iT_j=T_jT_i,\,\,&\forall\,1\leq j<i-1<r-1;\\
 & T_iT_{i+1}T_i=T_{i+1}T_iT_{i+1},\,\,&\forall\,1\leq i<r-1;\\
 & X_iX_i^{-1}=1=X_i^{-1}X_i,\,\, X_iX_j=X_jX_i,\,&\forall\,1\leq i,j\leq r;\\
 & T_iX_iT_i=\up^2 X_{i+1},\,\,X_jT_i=T_iX_j.\,\,&\forall\,1\leq i<r, j\notin\{i,i+1\}.
\endaligned$$

\subsection{Affine quantum Schur algebras and affine tensor spaces}
We next introduce the affine quantum Schur algebra and then describe its tensor space realization.

For $n\geq 1$, let $\afmbnn=\{(\la_i)_{i\in\mbz}\mid \la_i\in\mbn,\,\la_i=\la_{i-n}\ \text{for}\ i\in\mbz\}.$
For $r\geq 0$, let
$\afLanr=\{\la\in\afmbnn\mid\sg(\la):=\sum_{1\leq i\leq n}\la_i=r\}.$
For $\la\in\La_\vtg(n,r)$, since $\la$ is completely determined by $(\la_1,\ldots,\la_n)$, we associate to $\la$ a standard Young subgroup $\fS_\la:=\fS_{(\la_1,\ldots,\la_n)}$ of the symmetric group $\fS_r$.
For a finite subset $X\han\affSr$, let $$T_X=\sum_{x\in
X}T_x\in\afHr\;\;\text{ and }\;\; x_\la=T_{\fS_\la}.$$
The following endomorphism algebras over $\sZ$ or $\mbq(\up)$
$$\sS_\vtg(n,r):=\End_{\afHr}\biggl
(\bop_{\la\in\La_\vtg(n,r)}x_\la\afHr\biggr)\,\text{ and }\,\afbfSr:=\End_{\afbfHr}\biggl
(\bop_{\la\in\La_\vtg(n,r)}x_\la\afbfHr\biggr)$$
are called {\it affine quantum Schur algebras} or, more specifically, {\it affine $\up$-Schur algebras} (cf. \cite{GV,Gr99,Lu99}).
Note that $\afbfSr\cong\afSr\ot_\sZ\mbq(v)$.

We shall also use the following standard notation for the minimal coset representatives associated with the Young subgroups.
For $\la\in\afLanr$, denote the set of shortest representatives of right cosets of $\fS_\la$ in $\affSr$ by
$$\afmsD_\la=\{d\in\affSr\mid \ell(wd)=\ell(w)+\ell(d)\text{ for
$w\in\fS_\la$}\}.$$
Note that elements in $\afmsD_\la$ can be characterised as follows:
\begin{equation}\label{minimal coset representative}
\aligned
d^{-1}\in\afmsD_\la
&\iff d(\la_{0,i-1}+1)<d(\la_{0,i-1}+2)<\cdots<d(\la_{0,i-1}+\la_i),\,\forall 1\leq i\leq n,\endaligned
\end{equation}
where $\la_{0,i-1}:=\sum_{1\leq t\leq i-1}\la_t$.

We next recall the tensor space realization of the affine quantum Schur algebra.
Let $\Og$ be the free $\sZ$-module with basis $\{\og_i\mid
i\in\mbz\}$. For
$\bfi=(i_1,\ldots,i_r)\in\mbz^r$, write
$\og_\bfi=\og_{i_1}\ot\og_{i_2}\ot\cdots\ot \og_{i_r}=
\og_{i_1}\og_{i_2}\cdots \og_{i_r}\in\Og^{\ot r}.$
According to \cite{VV99},
$\Og^{\ot r}$
admits a right $\afHr$-module structure defined by
\begin{equation*}\label{afH action}
\begin{cases}
\og_{\bf i}\cdot X_t^{-1}
=\og_{i_1}\cdots\og_{i_{t-1}}\og_{i_t+n}\og_{i_{t+1}}\cdots\og_{i_r},\qquad \text{ for all }\bfi\in \mbz^r;\\
{\og_{\bf i}\cdot T_k=\left\{\begin{array}{ll} \up^2\og_{\bf
i},\;\;&\text{if $i_k=i_{k+1}$;}\\
v\og_{\bfi s_k},\;\;&\text{if $i_k<i_{k+1}$;}\qquad\text{ for all }\bfi\in I(n,r),\\
v\og_{\bfi s_k}+(\up^2-1)\og_{\bf i},\;\;&\text{if
$i_{k+1}<i_k$,}
\end{array}\right.}
\end{cases}
\end{equation*}
where $1\leq k\leq r-1$ and $1\le t\le r$.

Let $\mpk$ be a commutative ring containing an invertible element $\vep$.  We will regard $\mpk$ as a  $\sZ$-module by specializing
$\up$ to $\vep$. We define $\afHrk=\afHr\ot_\sZ \mpk$, and call it the extended affine Hecke algebra of type $A$ with Hecke parameter $\vep$. Furthermore let
$\afSrk=\afSr\ot_\sZ \mpk$,  and $\Ogk=\Og\ot_\sZ\mpk$.
 Then
the right action of $\afHr$ on $\Og^{\ot r}$ induces to a right
action of $\afHrk$ on $\Ogk^{\ot r}$.
According to \cite[Lem. 8.3]{VV99} (cf. \cite[Prop. 3.2.8 and 3.3.1]{DDF}), we have
\begin{equation}\label{T decom}
\Ogk^{\ot r}\cong
\bigoplus_{\lambda\in\afLanr}x_\lambda\afHrk.
\end{equation}
Hence we have
$$\afSr\cong\End_{\afHr}(\Og^{\ot r}),\quad \afSrk\cong\End_{\afHrk}(\Ogk^{\ot r}).$$  We will identify $\afSr$ and $\afSrk$ with $\End_{\afHr}(\Og^{\ot r})$
and $\End_{\afHrk}(\Ogk^{\ot r})$, respectively.
The two actions of $\afSrk$ and $\afHrk$ on $\Ogk^{\otimes r}$ commute. Hence the $\afSrk$-$\afHrk$-bimodule structure induces the natural $\mpk$-algebra homomorphism
\begin{equation}\label{xir}
\aligned
\xr :&\ \afHrk\lra\End_{\afSrk}(\Ogk^{\otimes r})^\oop.
\endaligned
\end{equation}
The injectivity of this homomorphism will be established in the next section. The remaining double-centralizer property will be obtained after studying the height-one local structure of the extended affine Hecke algebra.
\section{The faithfulness of the Hecke action}

In this section we prove that the natural homomorphism
$\xr$ from $\afHrk$ to $\End_{\afSrk} \Bigl( \bigoplus_{\lambda \in \afLanr} x_\lambda \afHrk \Bigr)^\oop$
is always injective by using the length function of the extended affine symmetric group. Throughout this section, $\mpk$ is a commutative ring containing an invertible element $\vep$. $\afHrk$ is the extended affine Hecke algebra associated to $\affSr$ over $\mpk$.

\subsection{Length estimates on the extended affine symmetric group}
We first establish several properties of the length function on the
extended affine symmetric group. These estimates will be used to
control the highest-length terms appearing in the proof of injectivity.

We begin with the basic multiplication formula when the length decreases.
\begin{lemma}\label{Lem1}
Let $u, w \in \affSr$. Suppose $\ell(u) + \ell(w) > \ell(uw)$. Then
$
T_u T_w$ is a $\mpk$-linear combination of  $ T_{u'}$ for $u'\in\affSr$ with $l(u') < l(u) + l(w)$.
\end{lemma}
\begin{proof}
We proceed by induction on $\ell(u)$.
If $\ell(u)=1$ then $u =\rho^a s_i$ for some $1\leq i\leq r$ and $a\in\mbz$. By definition we have
\[
T_{u} T_w = (\vep^{2}-1) T_{\rho^a w} + \vep^2 T_{\rho^a s_i w}.
\]
Now assume $\ell(u) > 1$. Write $u = s_i x$ with $\ell(u) = \ell(x) + 1$ and suppose $\ell(uw) < \ell(u) + \ell(w)$.

 {Case 1:} $\ell(xw) =\ell(x) +\ell(w)$. Then
\[
\ell(uw) = \ell(s_i x w)<\ell(s_ix)+\ell(w)= 1 + \ell(xw).
\]
It follows that
\[
T_u T_w = T_{s_i} T_{x w} = (\vep^{2}-1) T_{x w} + \vep^2 T_{s_i x w}
= (\vep^2-1) T_{x w} + \vep^2 T_{u w}.
\]
Hence we have $T_u T_w\in\spann\{T_{u'}\mid u'\in\affSr,\, \ell(u') < \ell(u) + \ell(w)\}$.

 {Case 2:} $\ell(xw) <\ell(x) +\ell(w)$. Then by the induction hypothesis
\[
T_x T_w =\sum_{u'\in\affSr\atop\ell(u') < \ell(x) + \ell(w)} a_{u'} T_{u'},
\]
where $a_{u'}\in\mpk$.
It follows that
\[
T_u T_w = T_{s_i} T_x T_w = \sum_{u'\in\affSr\atop\ell(u') < \ell(x) + \ell(w)} a_{u'} T_{s_i} T_{u'}\in\spann_{\mpk}\{T_{u''}\mid u''\in\affSr,\,\ell(u'') < \ell(u) + \ell(w)\}.
\]
The proof is completed.
\end{proof}

The preceding lemma describes the terms produced by a length decrease.
We next construct elements which increase the length additively.

\begin{lemma}\label{Lem2}
Let $w \in \affSr$ and suppose
\[
w^{-1}(i_1) < w^{-1}(i_2) < \cdots < w^{-1}(i_r),
\]
where $\{i_1, i_2, \dots, i_r\} = \{1, 2, \dots, r\}$.
Define $u\in\affSr$ by $u(k)=i_k$ for $1\leq k\leq r$.
Assume $x \in\fSr$ is such that $xu$ is the longest element in $\fSr$ and $\ell(xu)=\ell(x)+\ell(u)$. Then
\[
w^{-1} x^{-1}(1) > w^{-1} x^{-1}(2) > \cdots > w^{-1} x^{-1}(r).
\] and $
\ell(x w) = \ell(x) + \ell(w) $.
\end{lemma}
\begin{proof}
Since $xu$ is the longest element in $\fSr$ we have $x^{-1}(k)=i_{r+1-k}$ for $1\leq k\leq r$. Hence we have
$$w^{-1}x^{-1}(k)=w^{-1}(i_{r+1-k})>w^{-1}(i_{r-k})=w^{-1}x^{-1}(k+1)$$
for $1\leq k<r$.
Let $x = s_{j_1}s_{j_2} \cdots s_{j_m}$ be a reduced expression with $1 \le j_1, j_2, \dots, j_m < r$. Then we have $j_m = i_k$ and $j_m+1 = i_{l}$ for some $1\leq k,l\leq r$.
By \cite[Cor.~4.2.3]{Shi86}, the inequality  $\ell(s_{j_m} u) > \ell(u)$ implies $k=u^{-1}(j_m) < u^{-1}(j_m+1)=l$.
It follows that
$
w^{-1}(i_k) < w^{-1}(i_{l}).$
Hence
$$\ell(s_{j_m} w) > \ell(w).$$
There exist $1\leq k_t,l_t\leq r$ such that  $s_{j_m}\cdots s_{j_{t+1}}(j_t)=i_{k_t}$ and  $s_{j_m}\cdots s_{j_{t+1}}(j_t+1)=i_{l_t}$.
Since $\ell(xu)=\ell(x)+\ell(u)$, we have
\[
\ell(s_{j_t}s_{j_{t+1}} \cdots s_{j_m} u) > \ell(s_{j_{t+1}} \cdots s_{j_m} u).
\]
It follows that
\[
k_t=u^{-1}(i_{k_t})=(s_{j_{t+1}} \cdots s_{j_m} u)^{-1}(j_t) <(s_{j_{t+1}} \cdots s_{j_m} u)^{-1}(j_t+1)=u^{-1}(i_{l_t})={l_t}.
\]
Therefore,
\[
w^{-1}s_{j_m}\cdots s_{j_{t+1}}(j_t)=w^{-1}(i_{{k_t}})  < w^{-1}(i_{l_t})=w^{-1}
s_{j_m}\cdots s_{j_{t+1}}(j_t+1).
\]
Thus,  for $1 \le t < m$,
\[
\ell(s_{j_t} s_{j_{t+1}} \cdots s_{j_m} w) > \ell(s_{j_{t+1}} \cdots s_{j_m} w).
\]
Consequently,  we have $
\ell(xw) = \ell(x) + \ell(w).$
\end{proof}

By \cite[(3.2.1.1)]{DDF} we have
\begin{equation}\label{length}
 \ell(w)=|\{(i,j)\in\mbz^2\mid 1\leq i\leq r,\,i<j,\,w(i)>w(j)\}|
\end{equation}
for $ w \in \affSr$. We first consider a special class of affine permutations whose action
on the first $r$ positions is decreasing.
\begin{lemma}\label{Lem3}
Let $ w \in \affSr$ such that
$
w^{-1}(1) > w^{-1}(2) > \cdots > w^{-1}(r).
$ Let $d\in\affSr$ be defined by $d(k)=r+1-k+(k-r)r$ for $1\leq k\leq r$.
Then
\[
\ell(w_{0,r} d w) = \ell(w_{0,r}) + \ell(d) + \ell(w),
\]
where $w_{0,r}$ is the longest element of the symmetric group $\fSr$.
\end{lemma}
\begin{proof}
Since
$d^{-1}(k) = (r - k + 1) + (k - 1)r$ for $1 \leq k \leq r$
we have $$w^{-1}d^{-1}(k) = w^{-1}(r - k + 1) + (k - 1)r$$ for $1 \leq k \leq r.$
Hence, since  $w^{-1}(1) > w^{-1}(2) > \cdots > w^{-1}(r)$,
we have
\[
w^{-1}d^{-1}(1) < w^{-1}d^{-1}(2) < w^{-1}d^{-1}(3) < \cdots < w^{-1}d^{-1}(r-1) < w^{-1}d^{-1}(r).
\]
Hence by \eqref{minimal coset representative} we have
$
 \ell(w_{0,r} dw) = \ell(w_{0,r}) + \ell(dw).$
It remains to show that
$$\ell(dw)=\ell(d)+\ell(w).$$
Let
\[
\begin{split}
A&=\{(i,j)\in\mbz^2\mid 1\leq i\leq r,\,i<j,\,w^{-1}d^{-1}(i)>w^{-1}d^{-1}(j)\},\\
B&=\{(i,j)\in\mbz^2\mid 1\leq i\leq r,\,i<j,\,d^{-1}(i)>d^{-1}(j)\}.
\end{split}
\]
By \eqref{length} we have
\begin{equation}\label{dw}
\ell(dw)-\ell(d)=|A|-|B| .
\end{equation}
By the definition of $d$ we have
\begin{equation}\label{B}
 B =\big\{ (b, sr + k) \,\big|\,  1 \leq s \leq r - 2,\  k
 \geq 1, \  k + s + 1 \leq b \leq r \big\}.
 \end{equation}
Let $1 \leq s \leq r-2$, $k \geq 1$, $k+s+1 \leq b \leq r$. Then
\[
d^{-1}(b)  = -b+1 + br,\quad
d^{-1}(sr+k)   = 1-k + (s+k)r.
\]
Hence
\[
d^{-1}(b) - d^{-1}(sr+k) = (b-s-k)r - (b-k) \geq r - (b-k) \geq k \geq 1
\]
and
\[
w^{-1}d^{-1}(b) = w^{-1}(r-b+1) + (b-1)r
\]
\[
w^{-1}d^{-1}(sr+k) = w^{-1}(r+1-k) + (s+k-1)r.
\]
Since $k \geq 1$ and $k+s+1 \leq b \leq r$, we have
$1 \leq r-b+1,\ r-k+1 \leq r$. Furthermore
since $b \geq k+s+1 > k$, we have $r-b+1 < r+1-k$. Therefore
\[
w^{-1}(r-b+1) > w^{-1}(r-k+1).
\]
Because $b \geq k+s+1$, we get $b-1 \geq s+k > s+k-1$, hence
$
(b-1)r > (s+k-1)r.$
Therefore we have
\[
w^{-1}d^{-1}(b) = w^{-1}(r-b+1) + (b-1)r   > w^{-1}(r-k+1) + (s+k-1)r = w^{-1}d^{-1}(sr+k).
\]
Therefore we have
$B\subseteq A$.
It follows that
\[
A \backslash B = \{ (i,j)\in\mbz^2 \mid 1 \leq i \leq r,\ i<j,\ d^{-1}(i) < d^{-1}(j),\
w^{-1}d^{-1}(i) > w^{-1}d^{-1}(j) \}.
\]
Let
\[
\begin{split}
C &= \{ (k,l)\in\mbz^2 \mid 1 \leq k \leq r,\ k < l,\  w^{-1}(k) > w^{-1}(l)\},\\
C' &= \{ (k,l)\in\mbz^2  \mid 1 \leq k \leq r,\ k < l,\ d(k) < d(l),\ w^{-1}(k) > w^{-1}(l)\} \end{split}
\]
Now consider the map
\[
\begin{split}
f:&A \backslash B \xrightarrow{f} C'  \\
&(i,j) \mapsto (k,l)
\end{split}
\]
where  $d^{-1}(i) = k + sr$ for some $1 \leq k \leq r$ and $s \in \mathbb{Z}$, and
$l = d^{-1}(j) - sr$.
It is easy to see that the map $f$ is bijective.
Therefore by \eqref{dw} we have
\begin{equation}\label{dw'}
 \ell(dw)-\ell(d)=|A|-|B|=|C'|.
\end{equation}
Observe that  $C'\han C$ and
\[
C \backslash C' = \{ (k,l)\in\mbz^2 \mid 1 \leq k \leq r,\ k < l,\ d(k) > d(l),\ w^{-1}(k) > w^{-1}(l) \}.
\]
Let $$X=\{ (k,l)\in\mbz^2 \mid 1 \leq k \leq r,\ k < l,\ d(k) > d(l)\}.$$ Then by the definition of $d$ and \eqref{B} we have
$$X=\big\{ (b, sr + k) \,\big|\,  1 \leq s \leq r - 2,\  k
 \geq 1, \  k + s + 1 \leq b \leq r \big\}=B.$$
 Let $1 \leq s \leq r-2$, $k \geq 1$, $k+s+1 \leq b \leq r$.
Since $1 \leq k < k+s+1 \leq b \leq r$, we have
$
w^{-1}(k) > w^{-1}(b).$
Hence
\[
w^{-1}(b) - w^{-1}(sr + k) = w^{-1}(b) - w^{-1}(k) - sr<0.
\]
This implies the equality of the following sets:
\[
\begin{aligned}
X
&= \{(k,\ell)\in\mbz^2 \mid 1 \leq k \leq r,\ k < \ell,\ d(k) > d(\ell),\ w^{-1}(k) < w^{-1}(\ell)\}.
\end{aligned}
\]
Therefore, we conclude:
\[
C \setminus C' = \varnothing.
\]
Hence $C=C'$. Consequently by \eqref{dw'} we have $\ell(dw)-\ell(d)=|C|=\ell(w)$. The proof is completed.
\end{proof}
Combining the previous construction with the preceding reduction,
we obtain the following length-additivity result for arbitrary affine
permutations.
\begin{lemma}\label{Lem4}
Let $ w \in \affSr $. Then there exists $ d \in \affSr $ such that
\[
\ell(\wo d w) = \ell(\wo) + \ell(d) + \ell(w).
\]
\end{lemma}
\begin{proof}
By Lemma \ref{Lem2}, there exists $ x \in \fSr $ such that
$
\ell(x w) = \ell(x) + \ell(w)
$
and
\[
(x w)^{-1}(1) > (x w)^{-1}(2) > \cdots > (x w)^{-1}(r).
\]
By Lemma \ref{Lem3}, there exists $ d \in \affSr $ such that
\[
\begin{aligned}
\ell(\wo d x w) &= \ell(\wo) + \ell(d) + \ell(x w) \\
&= \ell(\wo) + \ell(d) + \ell(x) + \ell(w).
\end{aligned}
\]
We claim that $ \ell(d x) = \ell(d) + \ell(x) $. Suppose, for contradiction, that $ \ell(d x) < \ell(d) + \ell(x) $. Then
\[
\begin{aligned}
\ell(\wo d x w) &\leq \ell(\wo) + \ell(d x) + \ell(w) \\
&< \ell(\wo) + \ell(d) + \ell(x) + \ell(w),
\end{aligned}
\]
which contradicts the previous equality. Hence $ \ell(d x) = \ell(d) + \ell(x) $, and we conclude
\[
\ell(\wo d x w) = \ell(\wo) + \ell(d x) + \ell(w).
\]
The proof is completed.
\end{proof}

\subsection{Faithfulness of the right action}
We now apply the preceding length estimates to prove the injectivity
of the natural homomorphism $\xr$.
\begin{theorem}\label{injection}
The algebra homomorphism $$\xr : \ \afHrk\lra\End_{\afSrk} \Bigl( \bigoplus_{\lambda \in \afLanr} x_\lambda \afHrk \Bigr)^\oop$$ defined in \eqref{xir} is injective.
\end{theorem}
\begin{proof}
Let $\sum_{w \in \affSr} a_w T_w \in \ker \xr$.
Then $\X \afHrk \sum_{w \in \affSr} a_w T_w = 0$. In particular, for any $w' \in \affSr$,
\begin{equation}\label{eq1 of Theorem}
\X T_{w'} \sum_{w \in \affSr} a_w T_w = 0.
\end{equation}
Let $I = \{ w \in \affSr \mid a_w \neq 0 \}$. If $\sum_{w \in \affSr} a_w T_w \neq 0$, then $I \neq \emptyset$. Let
\[
m = \max \{ \ell(w) \mid w \in I \},
\]
and choose $y \in \affSr$ such that $a_y \neq 0$ and $\ell(y) = m$. By Lemma \ref{Lem4}, there exists $u \in \affSr$ such that
\begin{equation}\label{eq2 of Theorem}
\ell(\wo u y) = \ell(\wo) + \ell(u) + \ell(y),
\end{equation}
where $\wo$ is the longest element of $\fSr$. By \eqref{eq1 of Theorem} we have
\begin{equation}\label{eq3 of Theorem}
0=\X T_u \sum_{w \in I} a_w T_w .
\end{equation}
Furthermore by Lemma \ref{Lem1} we have
\[
\begin{aligned}
\X T_u \sum_{w \in I} a_w T_w& =  \Bigl( T_{\wo} + \sum_{\substack{w' \in \fSr \\ \ell(w') < \ell(\wo)}} T_{w'} \Bigr) T_u \Bigl( \sum_{\substack{w \in I \\ \ell(w) = m}} a_w T_w + \sum_{\substack{w \in I \\ \ell(w) < m}} a_w T_w \Bigr)\\
&= T_{\wo} T_u \sum_{\substack{w \in I \\ \ell(w) = m}} a_w T_w + \sum_{\substack{u' \in \affSr \\ \ell(u') < \ell(\wo) + \ell(u) + m}} b_{u'} T_{u'}
\end{aligned}
\]
where $b_{u'}\in\mpk$.
Hence by \eqref{eq2 of Theorem} we have
\[
\begin{aligned}
\X T_u \sum_{w \in I} a_w T_w&= a_y T_{\wo u y} + \sum_{\substack{w \in I,\,w \neq y\\ \ell(w) = m }} a_w T_{\wo}T_uT_w + \sum_{\substack{u' \in \affSr \\ \ell(u') < \ell(\wo u y)}} b_{u'} T_{u'}\\
&=
a_y T_{w_{0,r}uy} + \sum_{\substack{w \in I,\,w \neq y\\ \ell(w) = m }} a_w T_{w_{0,r}u} T_w + \sum_{\substack{u' \in \affSr \\ \ell(u') < \ell(w_{0,r}uy)}}b_{u'} T_{u'}\\
&=a_y T_{w_{0,r}uy} +f_1+f_2 + \sum_{\substack{u' \in \affSr \\ \ell(u') < \ell(w_{0,r}uy)}} b_{u'} T_{u'}
\end{aligned}
\]
where
\begin{equation*}
\begin{split}
 f_1&= \sum_{\substack{w \in I,\, \ell(w) = m,\, w \neq y \\ \ell(w_{0,r}uw) = \ell(w_{0,r}u) + \ell(w)   = \ell(w_{0,r}uy)}} a_w T_{w_{0,r}uw}\\
 f_2& =\sum_{\substack{w \in I,\, \ell(w) = m,\, w \neq y \\ \ell(w_{0,r}uw) < \ell(w_{0,r}u) + \ell(w)}} a_w T_{w_{0,r}u} T_w.
\end{split}
\end{equation*}
By Lemma \ref{Lem1} we have
\[
\begin{aligned}
f_2
&= \sum_{\substack{u' \in \affSr \\ \ell(u') < \ell(w_{0,r}u) + m=\ell(w_{0,r}uy)}} c_{u'} T_{u'}
\end{aligned}
\]
where $c_{u'}\in\mpk$.
Thus by \eqref{eq3 of Theorem} we have

\[
\begin{aligned}
&0=\ a_y T_{w_{0,r}uy} + \sum_{\substack{w\in I,\, \ell(w) = m,\,w \neq y \\ \ell(w_{0,r}uw) = \ell(w_{0,r}uy)}} a_w T_{w_{0,r}uw} + \sum_{\substack{u' \in \affSr \\ \ell(u') < \ell(w_{0,r}uy)}} (b_{u'}+c_{u'})  T_{u'}  = a_y T_{w_{0,r}uy} + g
\end{aligned}
\]
where $g$ is a $\mpk$-linear combination of $T_{u'}$ for $ u' \in \affSr$ with  $u' \neq w_{0,r}uy$. Therefore $a_y = 0$. This is a contradiction.
Consequently, $\ker\xr=0$.
\end{proof}

\section{ An abstract double-centralizer criterion}
In this section we formulate the abstract double-centralizer argument for subsequent use. We first give a concrete criterion for a homomorphism $H\longrightarrow T^{\oplus a}$ to be a left $\add(T)$-approximation. We then recall the double-centralizer criterion stated in terms of two successive approximations. Finally, we derive a convenient reformulation of this criterion, reducing the problem to constructing an embedding of the cokernel into a finite direct sum of copies of $T$.

Let $R$ be a Noetherian domain, let $H$ be an $R$-algebra, and let $T$ be a
right $H$-module. Put
$
S=\End_H(T).$
For every right $H$-module $M$, set
\[
M^*=\Hom_H(M,T),\qquad
M^{**}=\Hom_S(M^*,T).
\]
There is a natural evaluation homomorphism
\[
\alpha_M:M\longrightarrow M^{**},
\qquad
\alpha_M(m)(\varphi)=\varphi(m).
\]
We say that $H$ has the double-centralizer property relative to $T$ if the natural evaluation homomorphism $\alpha_H$ is an isomorphism, i.e.,
\[
\alpha_H: H\xrightarrow{\sim}H^{**}.
\]

Let $\add(T)$ denote the full subcategory of right $H$-modules consisting of
direct summands of finite direct sums of copies of $T$.

A homomorphism of right $H$-modules
$
f:M\longrightarrow C$
with $C\in\add(T)$ is called a left $\add(T)$-approximation of $M$ if, for
every $D\in\add(T)$, the map
\[
\Hom_H(C,D)\longrightarrow\Hom_H(M,D),
\qquad
g\longmapsto g\circ f,
\]
is surjective.

\begin{lemma}\label{left approx}
Let
$f:H\rightarrow T^{\oplus a}$
be an injective homomorphism of right $H$-modules, and write
$f(1)=(t_1,\ldots,t_a).$
Then $f$ is a left $\add(T)$-approximation of $H$ if and only if
$T=\sum_{i=1}^a St_i.$
\end{lemma}

\begin{proof}
By definition,
$f$ is a left $\add(T)$-approximation of $H$ if and only if the map
\[
\Hom_H( T^{\oplus a},T)\longrightarrow\Hom_H(H,T),
\qquad
g\longmapsto g\circ f,
\]
is surjective.
Using the natural identifications
$
\Hom_H(T^{\oplus a},T)\cong S^{\oplus a}$ and
$\Hom_H(H,T)\cong T,
$
the homomorphism induced by composition with $f$ becomes
\[
S^{\oplus a}\longrightarrow T,
\qquad
(s_1,\ldots,s_a)\longmapsto\sum_{i=1}^a s_i(t_i).
\]
Hence it is surjective if and only if
$
T=\sum_{i=1}^aSt_i.$
\end{proof}

We now recall Auslander-Solberg's criterion (\cite[Prop. 2.1]{AS}) for double centralizer property. Note that Auslander-Solberg's original result \cite[Prop. 2.1]{AS} considered only Artin algebras. We replace this Artin algebra assumption with a Noetherian assumption so that we can apply it in our infinite dimensional modules setting. Because of this modification, we include the full details of the proof for the convenience of the readers.

\begin{lemma}[{\cite[Prop. 2.1]{AS}}]\label{double-centralizer-criterion}
$(1)$
Assume that $S=\End_H(T)$ is left Noetherian. Suppose that $H$ has the double-centralizer property relative to $T$. If
$f: H\rightarrow T^{\oplus a}$ is injective and $f$
is a left $\add(T)$-approximation. Then there exist $b\geq0$ and a
homomorphism
$
g:T^{\oplus a}\longrightarrow T^{\oplus b}$
such that
$
0\rightarrow H\xrightarrow{\,f\,}T^{\oplus a}
\xrightarrow{\,g\,}T^{\oplus b}$
is exact and the induced monomorphism
$
\bar g:C:=T^{\oplus a}/f(H)\rightarrow T^{\oplus b}$
is a left $\add(T)$-approximation.

$(2)$ Suppose that there is an exact sequence
$
0\rightarrow H\xrightarrow{\,f\,}T^{\oplus a}
\xrightarrow{\,g\,}T^{\oplus b}$
such that $f$ and the induced map
$
\bar g:C:=T^{\oplus a}/f(H)\rightarrow T^{\oplus b}
$
are left $\add(T)$-approximations. Then $H$ has the double-centralizer property relative to $T$.
\end{lemma}

\begin{proof}
For (1), set $C=\Coker(f)$. We have
\begin{equation}\label{ses general}
0\longrightarrow H\xrightarrow{\,f\,}T^{\oplus a}
\xrightarrow{\,\pi\,}C\longrightarrow0.
\end{equation}
Since $f$ is a left $\add(T)$-approximation, we obtain the following short exact sequence
\begin{equation}\label{dual ses}
0\longrightarrow C^*
\xrightarrow{\,\pi^*\,}
(T^{\oplus a})^*
\xrightarrow{\,f^*\,}
H^*
\longrightarrow0.
\end{equation}
Now
$(T^{\oplus a})^*\cong S^{\oplus a}$, $H^*\cong T.$
Hence $C^*$ is a submodule of the finitely generated $S$-module
$S^{\oplus a}$. Since $S$ is Noetherian, $C^*$ is finitely generated. Choose
a surjection
\begin{equation}\label{q surj}
h:S^{\oplus b}\twoheadrightarrow C^*.
\end{equation}

We claim that
$
\alpha_C:C\rightarrow C^{**}$
is injective. Applying $\Hom_S(-,T)$ to \eqref{dual ses} gives
\begin{equation}\label{double dual seq}
0\longrightarrow H^{**}
\longrightarrow (T^{\oplus a})^{**}
\longrightarrow C^{**}.
\end{equation}
By naturality of the evaluation maps, \eqref{ses general} and \eqref{double dual seq} fit into the corresponding commutative diagram.
By hypothesis, $\alpha_H$ is an
isomorphism, while
$
\alpha_{T^{\oplus a}}:T^{\oplus a}\xrightarrow{\sim}(T^{\oplus a})^{**}$
is an isomorphism. If $c=\pi(x)\in\ker\alpha_C$, then
$\alpha_{T^{\oplus a}}(x)$ belongs to the image of $H^{**}$. By commutativity
and bijectivity of $\alpha_H$, this implies $x\in f(H)$, and hence $c=0$.
Thus $\alpha_C$ is injective.

Dualizing \eqref{q surj} gives an injection
$h^*:C^{**}\hookrightarrow T^{\oplus b}.$
Define
\[
\bar g=h^*\circ\alpha_C,
\qquad
g=\bar g\circ\pi.
\]
Then $\ker g=f(H)$, so
\[
0\longrightarrow H\xrightarrow{\,f\,}T^{\oplus a}
\xrightarrow{\,g\,}T^{\oplus b}
\]
is exact.
Under the natural identifications, the map
$
\Hom_H(\bar g,T):S^{\oplus b}\rightarrow C^*$
is precisely $h$. Hence it is surjective, and therefore $\bar g$ is a left
$\add(T)$-approximation.

For (2), since $f$ and  $\bar g$  are left $\add(T)$-approximations,  the induced sequence $
0\rightarrow C^*
\rightarrow (T^{\oplus a})^*
\rightarrow H^*
\rightarrow0,$
is exact and  $\bar g$ induces  a surjection
$
(T^{\oplus b})^*\rightarrow C^*.$
Hence
$
(T^{\oplus b})^*\rightarrow (T^{\oplus a})^*
\rightarrow H^*\rightarrow0$
is exact. Applying $\Hom_S(-,T)$ we obtain the following exact sequence
\[
0\longrightarrow H^{**}
\longrightarrow (T^{\oplus a})^{**}
\longrightarrow (T^{\oplus b})^{**}.
\]
Comparison with the original exact sequence, using
$
T^{\oplus a}\cong(T^{\oplus a})^{**},$
$T^{\oplus b}\cong(T^{\oplus b})^{**}, $
gives
$
H\cong H^{**}.$
\end{proof}

For our later application, it is convenient to avoid constructing the second approximation directly. The following lemma shows that, when
$S$ is left Noetherian, it is enough to embed the cokernel of the first approximation into a finite direct sum of copies of
$T$.

\begin{proposition}\label{embedding criterion}
Assume that $S=\End_H(T)$ is left Noetherian and that there is an exact sequence
\begin{equation}\label{criterion exact}
0\longrightarrow H\xrightarrow{\,f\,}T^{\oplus a}
\xrightarrow{\,g\,}T^{\oplus b},
\end{equation}
where $f$ is a left $\add(T)$-approximation. Put
$
C=T^{\oplus a}/f(H).$ Then, there exists an integer $c\geq 0$ and an injective homomorphism
$
\widetilde g:C\rightarrow T^{\oplus(b+c)}$
which is a left $\add(T)$-approximation of $C$. Consequently,   $H$ has the
double-centralizer property relative to $T$.
\end{proposition}

\begin{proof}
  Let
$
\pi:T^{\oplus a}\twoheadrightarrow C$
be the canonical quotient map. Since $\pi$ is surjective, composition
with $\pi$ gives an injection
\[
C^* =\Hom_H(C,T)
\hookrightarrow
\Hom_H(T^{\oplus a},T)
\cong S^{\oplus a}.
\]
Since $S$ is left Noetherian and $S^{\oplus a}$ is a finitely
generated left $S$-module, $C^*$ is a finitely generated left $S$-module (see \cite[Ch. 2, Lem. 2.5]{GS}).
Therefore there exist an integer $c\geq0$ and a surjective
$S$-module homomorphism
\[
h:S^{\oplus c}\twoheadrightarrow C^*.
\]
Let $e_1,\ldots,e_c$ be the standard generators of $S^{\oplus c}$ and
write
\[
h(e_i)=\varphi_i\in C^*
\qquad(1\leq i\leq c).
\]
Define
\[
j_0:C\longrightarrow T^{\oplus c},
\qquad
j_0(x)=
\bigl(\varphi_1(x),\ldots,\varphi_c(x)\bigr).
\]
The map $j_0$ induces a map,
$
\Hom_H(j_0,T):
\Hom_H(T^{\oplus c},T)
\rightarrow
\Hom_H(C,T)=C^*.$ Let
$i_0: S^{\oplus c}\ra \Hom_H(T^{\oplus c},T)$ be the natural isomorphism.
Then we have $\Hom_H(j_0,T)\circ i_0=h$. Hence $\Hom_H(j_0,T)$ is surjective since $h$ is surjective. It follows that $j_0$ is a
left $\add(T)$-approximation of $C$.

Since \eqref{criterion exact} is exact, $g$ induces an injection
$
\bar g:C\hookrightarrow T^{\oplus b}.
$
Define
\[
\widetilde g
=
\bar g\oplus j_0:
C\longrightarrow
T^{\oplus b}\oplus T^{\oplus c}
\cong
T^{\oplus(b+c)}.
\]
Since $\bar g$ is injective, $\widetilde g$ is injective.
Moreover, $\widetilde g$ is a left $\add(T)$-approximation  since $j_0$ is a left
$\add(T)$-approximation.
Finally, since $\widetilde g$ is injective,
\[
0\longrightarrow H
\xrightarrow{\,f\,}
T^{\oplus a}
\xrightarrow{\,\widetilde g\circ\pi\,}
T^{\oplus(b+c)}
\]
is exact, and the induced map
$
C\hookrightarrow T^{\oplus(b+c)}$
is precisely $\widetilde g$, which is a left
$\add(T)$-approximation. Hence by
Lemma~\ref{double-centralizer-criterion} we conclude that $H$ has the
double-centralizer property relative to $T$.
\end{proof}

\section{Height-one analysis of the affine Hecke algebra}
In this section, we establish the height-one
localization property
$\msH_{\mathfrak p}e\msH_{\mathfrak p}=
\msH_{\mathfrak p}$
needed for the double-centralizer argument. We first collect the
necessary structural properties and then use principal-series
modules to prove  this equality.
\subsection{The affine tensor space and the Bernstein center}
We first establish the finiteness properties and Noetherianity required for subsequent arguments.

Let $\mpk$ be an arbitrary commutative ring
containing an invertible element $\vep$, and put
$q=\vep^2$.
Set
\[
\msH=\afHrk,\qquad
\msH_0=\Hrk,\qquad
\msT=\Ogk^{\otimes r},
\]
and put
\[
\msS=\End_{\msH}(\msT)=\afSrk,
\]
where $\Hrk$ is the subalgebra of $\afHrk$ generated by $T_i$ for $1\leq i\leq r-1$.

Let
\[
\msP=\mpk[X_1^{\pm1},\ldots,X_r^{\pm1}],
\qquad
\msZ=\msP^{\mathfrak S_r}.
\]
We first prove two finiteness properties of $\msH$ and $\msT$ over the Bernstein center $\msZ$ that hold over an arbitrary commutative ground ring.
\begin{lemma}\label{H is finite rank}
$\msH$ is a free $\msZ$-module of finite rank.
\end{lemma}
\begin{proof}
By the Bernstein center theorem, $\msZ$ is the center of $\msH$.
Furthermore by \cite[Rem. 2.14]{Rouquier},  $\msP$ is a free $\msZ$-module of rank $r!$. Hence, since
$
\msH=\bigoplus_{w\in\mathfrak S_r}T_w\msP
 $,
 $\msH$ is a free $\msZ$-module of finite rank.
\end{proof}
The tensor-space decomposition \eqref{T decom} gives a corresponding finite-freeness property for $\msT$.
\begin{lemma}\label{T finite free over Z}
The $\msZ$-module $\msT$ is free of finite rank. Furthermore, $\msT$ is  a finitely generated left $\msS$-module.
\end{lemma}

\begin{proof}
For $\lambda\in\afLanr$, let $\msD_\lambda$ be the set of
minimal-length representatives of the right cosets
$\mathfrak S_\lambda\backslash\mathfrak S_r$. Then we have
\[
x_\lambda\msH
=
\bigoplus_{d\in\msD_\lambda}
x_\lambda T_d\msP.
\]
 Thus $x_\lambda\msH$ is a free $\msP$-module of finite
rank. By \cite[Rem. 2.14]{Rouquier},
 $\msP$ is a free $\msZ$-module of finite rank.
Hence by
\eqref{T decom} we conclude that $\msT$ is a free
$\msZ$-module of finite rank.

The right action of $\msZ$ on $\msT$ induces a ring homomorphism
$
\rho_{\msZ}:\msZ\longrightarrow \End(\msT).$
Since $\msZ$ is the center of $\msH$, the image of $\rho_{\msZ}$ is contained
in $\msS$. Hence, since $\msT$ is a finite generated $\msZ$-module, $\msT$ is also
finitely generated as a left $\msS$-module.
\end{proof}

We next impose a Noetherian hypothesis on the ground ring in order to obtain the corresponding Noetherianity of $\msZ$ and $\msS$.

\begin{proposition}\label{finite generated}
Assume that $\mpk$ is a Noetherian commutative ring. Then  the ring $\msZ$ is a Noetherian ring.
Moreover,
$\msS$ is both left and right Noetherian.
\end{proposition}
\begin{proof}
By the fundamental theorem of symmetric polynomials, we have
\begin{equation}\label{symmetric polynomials}
\msZ=
\mpk[e_1,\ldots,e_{r-1},e_r^{\pm1}],
\end{equation}
where $
e_i=e_i(X_1,\ldots,X_r)$
denote the $i$th elementary symmetric polynomial in
$X_1,\ldots,X_r$. Hence by \cite[Prop. 7.3 and Cor. 7.6]{AM}, $\msZ$ is Noetherian.

By Lemma~\ref{T finite free over Z}, $\msT$ is   a finitely generated $\msZ$-module.
Choose
$\msZ$-generators
$
t_1,\ldots,t_m$
of $\msT$. Then the map
\[
\End_{\msZ}(\msT)\longrightarrow \msT^{\oplus m},
\qquad
\varphi\longmapsto
\bigl(\varphi(t_1),\ldots,\varphi(t_m)\bigr),
\]
is an injective $\msZ$-module homomorphism.
Since $\msZ$ is a Noetherian ring,
$\msT^{\oplus m}$ is a finitely generated
$\msZ$-module. It follows that $\End_{\msZ}(\msT)$ is a finitely
generated $\msZ$-module (see \cite[Ch. 2, Lem. 2.5]{GS}).
Since
$
\msS=\End_{\msH}(\msT)
\subseteq
\End_{\msZ}(\msT)
$
is an $\msZ$-submodule, $\msS$ is finitely generated over the Noetherian ring $\msZ$. Therefore,
every left ideal and every right ideal of $\msS$ is finitely
generated as an $\msZ$-module, and hence is finitely generated
as a left or right $\msS$-module, respectively. Therefore $\msS$ is both left and
right Noetherian (see \cite[Ch.~2, Lems.~2.3 and 2.5]{GS}).
\end{proof}
Note that by \cite[Rem.~1.7]{DY}, the above proposition implies that ${\mathcal S}_{\vtg}(n,r)_{\mathbb{Z}[v,v^{-1}]}$ is an affine quasi-hereditary $\mathbb{Z}[v,v^{-1}]$-algebra in the sense of \cite{Kl}.

When the ground ring is a unique factorization domain, the Bernstein center has the additional unique factorization property needed in the sequel.
\begin{lemma}\label{Z is UFD}
Assume that $\mpk$ is a unique factorization domain. Then the ring $\msZ$ is a unique factorization domain.
\end{lemma}
\begin{proof}
Since $\mpk$ is a unique factorization domain, $\mpk[e_1,\ldots,e_{r-1},e_r]$ is a unique factorization domain.
Hence by \eqref{symmetric polynomials},
$\msZ$  is also a unique factorization domain.
\end{proof}

\subsection{Height-one primes}
We now consider height-one primes of the Bernstein center.
The next lemma shows that, for a prime $\mathfrak q$ of $\msP$
whose contraction to $\msZ$ has height one, there exists at most one
unordered pair $\{i,j\}$ such that
$X_i=q^{\pm1}X_j$
in the fraction field of $\msP/\mathfrak q$.

\begin{lemma}\label{height one q link}
Assume that $\mpk$ is a unique
factorization domain.
Let
$
\mathfrak p\in\Spec \msZ,$ $\operatorname{ht}\mathfrak p=1
$.
Then there exists
$
\mathfrak q\in\Spec \msP$
such that
$
\mathfrak q\cap \msZ=\mathfrak p$.
Moreover, for every
$
\mathfrak q\in\Spec\msP$
satisfying
$
\mathfrak q\cap\msZ=\mathfrak p,
$
we have
$
\operatorname{ht}\mathfrak q=1,
$
and
there exists at most one unordered pair
$\{i,j\}\subseteq\{1,\ldots,r\}$, $i\neq j,$
such that
$X_i=q^{\pm1}X_j$ in $\kappa(\mathfrak q)$, where
$\kappa(\mathfrak q)=\Frac(\msP/\mathfrak q)$
is the fraction field of $\msP/\mathfrak q$.
\end{lemma}

\begin{proof}
By \cite[Rem. 2.14]{Rouquier},   $\msP$ is a finitely generated $\msZ$-module,
Hence by \cite[Prop. 5.1]{AM} $\msP$ is integral over $\msZ$.
Therefore by \cite[Th. 5.10]{AM},  there exists
$
\mathfrak q\in\Spec \msP$
such that
$
\mathfrak q\cap \msZ=\mathfrak p$.

Now let $\mathfrak q\in\Spec\msP$ be any prime satisfying
$
\mathfrak q\cap\msZ=\mathfrak p.$
By \cite[Cor. 5.9]{AM}
$
\operatorname{ht}\mathfrak q\leq\operatorname{ht}\mathfrak p=1.$
Since
$
\mathfrak q\cap \msZ=\mathfrak p\neq(0),$
we have $\mathfrak q\neq(0)$. As $\msP$ is a domain,
$
(0)\subsetneq\mathfrak q,$
so
$
\operatorname{ht}\mathfrak q\geq1.$
Therefore
$
\operatorname{ht}\mathfrak q=1.$

Since
$
\msP=\mpk[X_1^{\pm1},\ldots,X_r^{\pm1}]$
is a unique factorization domain and $
\operatorname{ht}\mathfrak q=1$,
$
\mathfrak q=(F)$
for an irreducible Laurent polynomial $F$.
Suppose
\[
X_i=q^\varepsilon X_j
\qquad
(\varepsilon\in\{1,-1\})
\]
in $\kappa(\mathfrak q)$ for some $i,j$. Since $\msP/\mathfrak q$ embeds into its fraction
field,
$
X_i-q^\varepsilon X_j\in\mathfrak q.$
Since
$
X_i-q^\varepsilon X_j
=
X_j(X_iX_j^{-1}-q^\varepsilon)$
is irreducible,
$
F\sim X_i-q^\varepsilon X_j.$
If there were another relation
\[
X_k=q^\eta X_l
\qquad
(\eta\in\{1,-1\}),
\]
then
$
F\sim X_k-q^\eta X_l.$
Thus
$
X_i-q^\varepsilon X_j
\sim
X_k-q^\eta X_l.$ It follows that
$
\{i,j\}=\{k,l\}.$  This proves the
assertion.
\end{proof}

\subsection{Principal-series modules}
From now on, let $\mpk$ be a field of characteristic $0$, let
$\vep\in\mpk^\times$, and put $q=\vep^2$. Assume that $\vep$
(or equivalently, $q$) is not a root of unity.

Recall that  $
x_\lambda=\sum_{w\in\mathfrak S_\lambda}T_w $ for
$
\lambda \in\afLanr$
By \cite[Lem. 3.1]{DF15} we have
\begin{equation*}\label{xla^2}
x_\lambda^2=c_\lambda x_\lambda,
\qquad
c_\lambda=
\sum_{w\in\mathfrak S_\lambda}q^{\ell(w)}
=
\prod_{i=1}^n[\lambda_i]_q!,
\end{equation*}
where
$
[j]_q=1+q+\cdots+q^{j-1}.$
Since $q$ is not a root of unity,  $c_\lambda\in \mpk^\times$ for every $\lambda$.
Let
\begin{equation}\label{ela}
e_\lambda=c_\lambda^{-1}x_\lambda.
\end{equation}
Then
$
e_\lambda
$
is an idempotent and
$
x_\lambda \msH=e_\lambda \msH.$
It follows from \eqref{T decom} that
\begin{equation}\label{tensor-decomp}
\msT\cong
\bigoplus_{\lambda\in\afLanr}x_\lambda \msH
\cong
\bigoplus_{\lambda\in\afLanr}e_\lambda \msH.
\end{equation}
Assume $n\geq 2$ and set
\begin{equation}\label{the element e}
\delta=(r-1,1,0,\ldots,0)\in\afLanr,
\qquad
e=e_\delta.
\end{equation}

Fix a height-one prime $\mathfrak p$ of $\msZ$ and a prime
$\mathfrak q$ of $\msP$ lying above it as in Lemma~\ref{height one q link}. We shall
study the principal-series modules associated with the corresponding
point and show that all of their simple composition factors have
nonzero $e$-part.

Let
$
\mathfrak p\in\Spec\msZ$ with $\operatorname{ht}\mathfrak p=1$.
By Lemma  \ref{height one q link} there exists
$
\mathfrak q\in\Spec \msP$
such that
$
\mathfrak q\cap \msZ=\mathfrak p$ and
$
\operatorname{ht}\mathfrak q=1.$ Let
\[
k=\overline{\Frac(\msP/\mathfrak q)}.
\]
Via the natural embedding $\mpk\hookrightarrow k$, we extend scalars
from $\mpk$ to $k$. For simplicity, we continue to write
$
\msH, \msH_0, \msP, e$
for the corresponding scalar extensions
$
k\otimes_{\mpk}\msH,
k\otimes_{\mpk}\msH_0,
k\otimes_{\mpk}\msP,
1\otimes e,$
respectively.

For any
$
b=(b_1,\ldots,b_r)\in(k^\times)^r,$
let $k_b$ denote the one-dimensional right $\msP$-module on which
$
X_i$
acts as multiplication by $b_i$, and define  the corresponding
principal-series module
\[
I(b)=k_b\otimes_{\msP}\msH.
\]

For $1\leq i<r$, set
\[
\Phi_i=(X_i-X_{i+1})T_i+(q-1)X_{i+1}.
\]
Then we have
\begin{equation}\label{Phi-intertwining}
\Phi_iX_i=X_{i+1}\Phi_i,\qquad
\Phi_iX_{i+1}=X_i\Phi_i,\quad \Phi_iX_j=X_j\Phi_i\qquad(j\neq i,i+1)
\end{equation}
and
\begin{equation}\label{Phi-square}
\Phi_i^2
=
(X_i-qX_{i+1})(X_{i+1}-qX_i)
\end{equation}
(cf. \cite[\S5]{Lu89}, \cite[\S4.2]{BK}).

The symmetric group $\frak S_r$ acts on $(k^\times)^r$ on the right by
$
(b\cdot w)_j=b_{w(j)}$ ($1\leq j\leq r$) for $b\in(k^\times)^r$ and $w\in\frak S_r$.
For $b\in (k^\times)^r$
we say that $b_i,b_j$ are $q$-linked if
$
b_i=q^{\pm1}b_j.$

\begin{lemma}\label{neighbor-iso}
Let
$
b=(b_1,\ldots,b_r)\in(k^\times)^r.$
If $b_i$ and $b_{i+1}$ are not $q$-linked for some
$1\leq i<r$, then
$
I(b)\cong I(b\cdot s_i)$
as right $\msH$-modules, where $s_i=(i,i+1)\in\mathfrak S_r $
is the simple transposition interchanging $i$ and $i+1$.
\end{lemma}

\begin{proof}
Put $
b'=b\cdot s_i,\
v_b=1\otimes1\in I(b),\
v_{b'}=1\otimes1\in I(b').$
By \eqref{Phi-intertwining},  there are
$\msH$-module homomorphisms
\[
F_i:I(b')\longrightarrow I(b),
\qquad
G_i:I(b)\longrightarrow I(b')
\]
such that
$
F_i(v_{b'})=v_b\Phi_i, \
G_i(v_b)=v_{b'}\Phi_i.$
By \eqref{Phi-square},
$
F_iG_i
=
c\,\id_{I(b)},
\
G_iF_i
=
c\,\id_{I(b')},
$
where
$
c=(b_i-qb_{i+1})(b_{i+1}-qb_i).
$
Since $b_i\neq q^{\pm1}b_{i+1}$, we have $c\neq0$. Thus $F_i$ is an
isomorphism, with inverse $c^{-1}G_i$. Therefore
$
I(b)\cong I(b')
=
I(b\cdot s_i).$
\end{proof}

Thus neighboring parameters which are not $q$-linked may
be interchanged without changing the isomorphism class of the
principal-series module. We first recall the irreducibility
criterion for the case in which no $q$-linked pair occurs.

Chari--Pressley \cite[Prop.~3.4(c)]{CP}, following
Rogawski \cite{R}, recall the following irreducibility criterion
over $\mathbb C$.  The same criterion is valid over any algebraically closed field of characteristic
$0$.

\begin{lemma}\label{CP charzero}
If
$
b=(b_1,\ldots,b_r)\in(k^\times)^r$
is such that
$
b_i\neq q^{\pm1}b_j
\ (i\neq j),$
then $I(b)$ is irreducible.
\end{lemma}
 In view of Lemmas \ref{height one q link} and \ref{CP charzero}, it remains to consider the case in which there exists exactly one unordered pair $\{i,j\}$ such that $b_i$ and $b_j$ are $q$-linked. For this purpose we
introduce two induced modules corresponding to the two possible
orientations of such a pair.

Let
$
\msH^{(1)}$
be the subalgebra of $\msH$ generated by $\msP$ and $T_1$.
For any
$
b=(b_1,\ldots,b_r)\in(k^\times)^r$
satisfying $
b_2=qb_1,$
let $k_b^+$ be the one-dimensional right $\msH^{(1)}$-module
defined by
$
X_j\mapsto b_j\ (1\leq j\leq r)$,
$
T_1\mapsto q,$
and set
\[
L^+(b)=k_b^+\otimes_{\msH^{(1)}}\msH.
\]
Similarly, for any
$
c=(c_1,\ldots,c_r)\in(k^\times)^r$
satisfying
$
c_1=qc_2,$
let $k_c^-$ be the one-dimensional right $\msH^{(1)}$-module
defined by
$
X_j\mapsto c_j\ (1\leq j\leq r),
\
T_1\mapsto -1,$
and set
\[
L^-(c)=k_c^-\otimes_{\msH^{(1)}}\msH.
\]
Then we have
\begin{equation}\label{Lpm dim}
\dim_k L^+(b)=\dim_k L^-(c)
=
[\mathfrak S_r:\langle s_1\rangle]
=
\frac{r!}{2}.
\end{equation}

Recall the element $e=e_\delta$ defined in \eqref{the element e}. The following lemma identifies the basic simple modules
arising in the one-link case and, crucially for our later
application, shows that their $e$-parts are nonzero.
\begin{lemma}\label{Lpm simple}
Let
$
b=(b_1,\ldots,b_r)\in(k^\times)^r$
have pairwise distinct entries,
and suppose that $\{1,2\}$ is the unique unordered pair of distinct indices $\{i,j\}$ such that $b_i$ and $b_j$ are $q$-linked.
\begin{enumerate}
\item[(1)]
If $b_2=qb_1$, then $L^+(b)$ is simple and
$
L^+(b)e\neq0.$

\item[(2)]
If $b_1=qb_2$, then $L^-(b)$ is simple and
$L^-(b)e\neq0.$
\end{enumerate}
\end{lemma}

\begin{proof}
We first prove the simplicity assertion in (1). Assume
$
b_2=qb_1$
and set
\[
\Gamma_+(b)=
\left\{
b\cdot w:
\text{$b_1$ occurs to the left of $b_2$ in $b\cdot w$}
\right\}.
\]
Since the entries of $b$ are pairwise distinct,
$
|\Gamma_+(b)|=\frac{r!}{2}.$

Let
$
v_+=1\otimes1\in L^+(b).$
For each $c\in\Gamma_+(b)$, one can move from $b$ to $c$ by a sequence of
adjacent transpositions
$
s_{i_1},\ldots,s_{i_m}$
which never reverse the relative order of $b_1$ and $b_2$.  Since $\{b_1,b_2\}$ is the unique
$q$-linked pair, every adjacent pair interchanged along such a path
is not $q$-linked. Hence the corresponding intertwiners are invertible
on the relevant weight spaces. We therefore obtain a nonzero vector
\[
 v_c:=v_+\Phi_{i_1}\cdots\Phi_{i_m}\in L^+(b)
\]
of $\msP$-weight $c$.

The weights in $\Gamma_+(b)$ are pairwise distinct, so the vectors
$v_c$ are linearly independent. By \eqref{Lpm dim} we have
\begin{equation}\label{Lplus-weight-decomp}
L^+(b)
=
\bigoplus_{c\in\Gamma_+(b)}kv_c.
\end{equation}
In particular, all these weight spaces are one-dimensional.

Let $
0\neq N\subseteq L^+(b)$
be an $\msH$-submodule. Since $\msP\subseteq\msH$, the subspace
$N$ is $\msP$-stable. As $k$ is algebraically closed and
$X_1,\ldots,X_r$ act on the finite-dimensional space $N$ as
commuting linear operators, they have a common eigenvector
$
0\neq w\in N.$
Thus $w$ is an $\msP$-weight vector. By
\eqref{Lplus-weight-decomp} we conclude that
$
w\in kv_c$
for some $c\in\Gamma_+(b)$, and therefore
$
v_c\in N.$ The set $\Gamma_+(b)$ is connected by adjacent
transpositions which do not reverse the relative order of $b_1$ and
$b_2$. Applying the corresponding intertwiners shows that
$
v_d\in N
\
\text{for every }d\in\Gamma_+(b).$
Hence
$
N=L^+(b),$
and therefore $L^+(b)$ is simple.

The same argument proves that $L^-(b)$ is simple when
$b_1=qb_2$, using
\[
\Gamma_-(b)=
\left\{
b\cdot w:
\text{$b_1$ occurs to the left of $b_2$ in $b\cdot w$}
\right\}.
\]

It remains to prove
\[
L^+(b)e\neq0
\qquad\text{and}\qquad
L^-(b)e\neq0
\]
in the respective cases.

Since $\{b_1,b_2\}$ is assumed to be a $q$-linked pair, we necessarily
have $r\geq2$.
If $r=2$, then
$
e=e_{(1,1)}=1.$ Thus the assertion is immediate. We may therefore assume that
$
r\ge3.$

We first consider $L^+(b)$, where
$
b_2=qb_1.$
Let $
\msH_2$ be the subalgebra of $\msH$ generated by $T_1.$
Then we have
\begin{equation*}
L^+(b)|_{\msH_0}
\cong
\mathbf1\otimes_{\msH_2}\msH_0,
\end{equation*}
as right $\msH_0$-modules,
where $\mathbf1$ is the one-dimensional right $\msH_2$-module on
which
$
T_1$
acts as multiplication by $q$.

Let
$
W_{r-1}=\langle s_1,\ldots,s_{r-2}\rangle$, $W_2=\langle s_1\rangle,$
and let $D$ be the set of minimal-length representatives of the
right cosets $W_2d$
in $W_{r-1}$. Let
$
v_+=1\otimes1\in \mathbf1\otimes_{\msH_2}\msH_0.$  Then by \eqref{ela} and \eqref{the element e} we have
\[ v_+e=
\frac{1}{[r-1]_q!}
v_+x_{(r-1,1)}
  = \frac{1}{[r-1]_q!}
\sum_{d\in D}\sum_{u\in W_2}v_+T_uT_d \notag =\frac{1}{[r-1]_q!}
(1+q)\sum_{d\in D}v_+T_d.
\]
Hence, since the vectors
$
v_+T_d,
\ d\in D,$
are distinct basis vectors of $\mathbf1\otimes_{\msH_2}\msH_0$, we have $v_+e\not=0$. Thus
$
L^+(b)e\neq0.$

We next consider $L^-(b)$, where
$
b_1=qb_2.$
Then we have
\begin{equation*}
L^-(b)|_{\msH_0}
\cong
\operatorname{sgn}\otimes_{\msH_2}\msH_0
\end{equation*}
as right $\msH_0$-modules, where $\operatorname{sgn}$ is the
one-dimensional right $\msH_2$-module on which
$
T_1$
acts as multiplication by $-1$.

 Let
$
v_-=1\otimes1\in \operatorname{sgn}\otimes_{\msH_2}\msH_0.$
Set
$
\tau=s_2s_3\cdots s_{r-1}.$
Since
$
\tau(1)<\tau(2)<\cdots<\tau(r-1)$,
$\tau$ is the minimal-length representative of
$
\tau W_{r-1}.$ Hence by \eqref{ela} and \eqref{the element e} we have
\begin{align}
(v_-T_\tau) e=
\frac{1}{[r-1]_q!}
(v_-T_\tau)x_{(r-1,1)}=\frac{1}{[r-1]_q!}
\sum_{u\in W_{r-1}}v_-T_\tau T_u \notag =\frac{1}{[r-1]_q!}
\sum_{u\in W_{r-1}}v_-T_{\tau u}.
\end{align}
For $u\in W_{r-1}$ we have
$
(\tau u)^{-1}(1)
=
u^{-1}(1)
\leq r-1
<
r
=
(\tau u)^{-1}(2).
$, and hence
 the element $\tau u$ is
the minimal-length representative of
$
W_2(\tau u).$
Therefore
the vectors
$
v_-T_{\tau u},
\
u\in W_{r-1},$
are distinct  basis vectors of $\operatorname{sgn}\otimes_{\msH_2}\msH_0$.
It follows that $(v_-T_\tau) e\not=0$.
Consequently, $
L^-(b)e\neq0.$
\end{proof}

We next relate these two simple modules to the original
principal-series module. In the presence of a unique $q$-linked
pair, they occur as its two composition factors.
\begin{lemma}\label{one-link-series}
Let
$
b=(b_1,\ldots,b_r)\in(k^\times)^r$
have pairwise distinct entries,
and suppose that $\{1,2\}$ is the unique unordered pair of distinct indices $\{i,j\}$ such that $b_i$ and $b_j$ are $q$-linked.
\begin{enumerate}
\item[(1)]
If $b_2=qb_1$, then there is a short exact sequence
\begin{equation}\label{ses plus}
0\longrightarrow L^-(b\cdot s_1)
\longrightarrow I(b)
\longrightarrow L^+(b)
\longrightarrow0.
\end{equation}

\item[(2)]
If $b_1=qb_2$, then there is a short exact sequence
\begin{equation}\label{ses minus}
0\longrightarrow L^+(b\cdot s_1)
\longrightarrow I(b)
\longrightarrow L^-(b)
\longrightarrow0.
\end{equation}
\end{enumerate}
\end{lemma}

\begin{proof}
Assume first that
$
b_2=qb_1.$
Put
$
w=1\otimes1\in I(b),
\
u=w(T_1-q).$
Then we have
\[
uX_1=b_2u,\qquad
uX_2=b_1u,\qquad
uX_j=b_ju\quad(j\neq1,2),\quad u(T_1+1)=0.
\]
Thus there is  an $\msH$-module homomorphism
$
\iota:
L^-(b\cdot s_1)\rightarrow I(b).$
such that $
\iota(1\otimes 1)=u
 $.
Since $
I(b)|_{\msH_0}\cong\msH_0$
and $T_1-q\neq0$ in $\msH_0$, we have $u\neq0$. Hence
$\iota\neq0$. By Lemma~\ref{Lpm simple},
$L^-(b\cdot s_1)$ is simple, so $\iota$ is injective.
On the other hand,   there is a
surjective $\msH$-module homomorphism
$
\pi:I(b)\twoheadrightarrow L^+(b)$
such that $\pi(w)=1\ot 1$. Clearly we have
$
\operatorname{Im}\iota
\subseteq
\ker\pi.$
Moreover,
$
\dim_k\operatorname{Im}\iota
=
\dim_kL^-(b\cdot s_1)
=
\frac{r!}{2},
$
whereas
$
\dim_k\ker\pi
=
\dim_kI(b)-\dim_kL^+(b)
=
r!-\frac{r!}{2}
=
\frac{r!}{2}.
$
Thus
$
\operatorname{Im}\iota=\ker\pi,$
and consequently
\[
0\longrightarrow L^-(b\cdot s_1)
\overset{\iota}{\longrightarrow}
I(b)
\overset{\pi}{\longrightarrow}
L^+(b)
\longrightarrow0
\]
is exact. This proves \eqref{ses plus}.

Now assume that
$
b_1=qb_2.$
Put $
w=1\otimes1\in I(b)$ and $
u'=w(T_1+1).$
Then we have
\[
u'X_1=b_2u',\qquad
u'X_2=b_1u',\qquad
u'X_j=b_ju'\quad(j\neq1,2),\quad u'T_1
=
qu'.
\]
Thus
there is an $\msH$-module homomorphism
$
\iota':
L^+(b\cdot s_1)\longrightarrow I(b) $
such that $\iota'(
1\otimes 1)-u'$.
Since
$
I(b)|_{\msH_0}\cong\msH_0$
and $T_1+1\neq0$ in $\msH_0$, we have $u'\neq0$. Hence
$\iota'\neq0$. By Lemma~\ref{Lpm simple},
$L^+(b\cdot s_1)$ is simple, so $\iota'$ is injective.
On the other hand,   there is a
surjective $\msH$-module homomorphism
$
\pi':I(b)\twoheadrightarrow L^-(b)$
such that
$
\pi'(w)=1\ot 1.$
Clearly we have
$
\operatorname{Im}\iota'
\subseteq
\ker\pi'.$
Moreover,
$
\dim_k\operatorname{Im}\iota'
=
\dim_kL^+(b\cdot s_1)
=
\frac{r!}{2},
$
whereas
$
\dim_k\ker\pi'
=
\dim_kI(b)-\dim_kL^-(b)
=
r!-\frac{r!}{2}
=
\frac{r!}{2}.
$
Hence
$
\operatorname{Im}\iota'=\ker\pi',$
and therefore
\[
0\longrightarrow L^+(b\cdot s_1)
\overset{\iota'}{\longrightarrow}
I(b)
\overset{\pi'}{\longrightarrow}
L^-(b)
\longrightarrow0
\]
is exact. This proves \eqref{ses minus}.
\end{proof}
Combining
Lemmas~\ref{Lpm simple} and \ref{one-link-series}, we obtain the following result for a principal-series module whose parameters are pairwise distinct and
whose unique $q$-linked pair occurs in the first two positions.

\begin{lemma}\label{one-link-composition-factors}
Let
$b=(b_1,\ldots,b_r)\in(k^\times)^r.$
Assume that the entries of $b$ are pairwise distinct,
and that $\{1,2\}$ is the unique
unordered pair of distinct indices $\{i,j\}$ such that $b_i$ and
$b_j$ are $q$-linked.
Then every simple composition factor $L$ of $I(b)$ satisfies
$Le\neq0.$
\end{lemma}

\begin{proof}
Since $b_1$ and $b_2$ are $q$-linked, either
$b_2=qb_1$ or $b_1=qb_2.$
Suppose first that
$b_2=qb_1.$  By
Lemma~\ref{one-link-series} there is a short exact sequence
$$
0\rightarrow L^-(b\cdot s_1)
\rightarrow I(b)
\rightarrow L^+(b)
\rightarrow0.$$
By Lemma~\ref{Lpm simple}, both $L^-(b\cdot s_1)$ and $L^+(b)$ are simple, and $L^-(b\cdot s_1)e\neq0$,
$L^+(b)e\neq0.$ Hence every simple composition factor $L$ of $I(b)$ satisfies
$Le\neq0$.

Suppose now that
$b_1=qb_2$. By
Lemma~\ref{one-link-series} there is a short exact sequence
$$
0\rightarrow L^+(b\cdot s_1)
\rightarrow I(b)
\rightarrow L^-(b)
\rightarrow0.$$
Again by Lemma~\ref{Lpm simple},
both $L^+(b\cdot s_1)$ and
$L^-(b)$ are simple, and
$
L^+(b\cdot s_1)e\neq0$, $L^-(b)e\neq0.$ Thus every simple composition factor $L$ of $I(b)$ satisfies
$Le\neq0.$
\end{proof}

We now return to the fixed primes
$\mathfrak p\in\Spec\msZ$, $\operatorname{ht}\mathfrak p=1,$
and
$\mathfrak q\in\Spec\msP$, $\mathfrak q\cap\msZ=\mathfrak p$, $\operatorname{ht}\mathfrak q=1.$
Recall that
$k=\overline{\Frac(\msP/\mathfrak q)}$.
Let
\begin{equation}\label{the element a}
a=(a_1,\ldots,a_r)\in(k^\times)^r
\end{equation}
be the point determined by $\mathfrak q$,
where $a_i\in k^\times$ is the image of $X_i$ under the natural map
\[
\msP\longrightarrow \msP/\mathfrak q
\longrightarrow \Frac(\msP/\mathfrak q)
\longrightarrow k.
\]

We now apply the preceding principal-series results to
the parameter $a$ determined by $\mathfrak q$.
This gives the principal-series statement needed
to prove $\msH_{\mathfrak p}e\msH_{\mathfrak p}=
\msH_{\mathfrak p}$ for height-one primes $\mathfrak p$.

\begin{lemma}\label{principal series main}
Every simple composition factor $L$ of $I(a)$ satisfies
$
Le\neq0.$
\end{lemma}

\begin{proof}
By Lemma~\ref{height one q link},
there exists at most one unordered pair of distinct indices
$\{i,j\}$ such that $a_i$ and $a_j$ are $q$-linked.
Suppose first that there is no $q$-linked pair. By
Lemma~\ref{CP charzero}, the principal-series module $I(a)$ is
simple. Moreover,
$
I(a)|_{\msH_0}\cong\msH_0$
as right $\msH_0$-modules. Hence
$
I(a)e\cong\msH_0e\neq0,$
and the assertion follows.

Suppose now that there exists exactly one unordered pair of distinct indices $\{i,j\}$ such that $a_i$ and $a_j$ are $q$-linked.  We first show that the entries
$a_1,\ldots,a_r$ are pairwise distinct. Write the unique $q$-link as
$
a_i=q^\varepsilon a_j,
\
\varepsilon\in\{1,-1\}.$
Then
$
X_i-q^\varepsilon X_j\in\mathfrak q.
$
If
$
a_s=a_t
\ (s\neq t),$
then
$
X_s-X_t\in\mathfrak q.$
By the proof of Lemma~\ref{height one q link},
$
\mathfrak q=(F)
$
for some irreducible $F\in\msP$. Since
$
X_s-X_t
\ \text{and}\
X_i-q^\varepsilon X_j
$
are irreducible elements of $\mathfrak q$, they are associate in
$\msP$. Hence
$
\{s,t\}=\{i,j\}.$
Thus $a_i=a_j$, and together with
$
a_i=q^\varepsilon a_j
$
this gives
$
q^\varepsilon=1,$
contrary to the assumption that $q$ is not a root of unity.
Therefore
$
a_1,\ldots,a_r$
are pairwise distinct.

Let $a_i,a_j$ be the unique $q$-linked pair. By repeatedly applying
Lemma~\ref{neighbor-iso} along adjacent transpositions which do
not interchange these two entries, we may choose
$
w\in\mathfrak S_r$
such that, for
$
b=a\cdot w,$
the linked pair occupies the first two positions and
$
I(a)\cong I(b).$
The entries of $b$ are pairwise distinct,
and $\{1,2\}$ is the unique
unordered pair of distinct indices $\{i,j\}$ such that $b_i$ and
$b_j$ are $q$-linked.
By Lemma~\ref{one-link-composition-factors}, every simple composition
factor $M$ of $I(b)$ satisfies
$Me\neq0.$
Since
$I(a)\cong I(b),$
the same is true for every simple composition factor $L$ of $I(a)$.
Therefore
$Le\neq0.$
\end{proof}

\subsection{Height-one localization}
We now apply the preceding principal-series analysis to the
height-one fibers of $\msH$. By Lemma~\ref{principal series main},
assuming that
$\msH_{\mathfrak p}e\msH_{\mathfrak p}
\neq
\msH_{\mathfrak p}$
for a height-one prime $\mathfrak p\in\Spec\msZ$ leads to a contradiction.

\begin{proposition} \label{prop:height-one-fullness}
For every
$\mathfrak p\in\Spec \msZ, \
\operatorname{ht}\mathfrak p=1,$
one has
$
\msH_{\mathfrak p}e\msH_{\mathfrak p}=\msH_{\mathfrak p}
$
where $e$ is defined in \eqref{the element e}.
\end{proposition}

\begin{proof}
Suppose
$\msH_{\mathfrak p}e\msH_{\mathfrak p}\neq \msH_{\mathfrak p}.$
Set
$
A=\msH_{\mathfrak p}/\msH_{\mathfrak p}e\msH_{\mathfrak p}.
$
Then $A\neq0$.
Put
$\mathfrak m=\mathfrak p\msZ_{\mathfrak p}.$
Since $\msH$ is a finite generated $\msZ$-module, $A$ is a finite generated
$\msZ_{\mathfrak p}$-module. Hence by Nakayama's lemma we have
$
A/\mathfrak mA=A/\mathfrak pA\neq0.$
(cf. \cite[Prop. 2.6]{AM}).
Let
$
\kappa(\mathfrak p)
=
\msZ_{\mathfrak p}/\mathfrak m,$
and let $k$ be an algebraic closure of $\kappa(\mathfrak p)$. Set
\[
\overline A
=
k\otimes_{\kappa(\mathfrak p)}A/\mathfrak mA.
\]
Since $A/\mathfrak mA$ is a nonzero finite-dimensional
$\kappa(\mathfrak p)$-vector space, $\overline A$ is a nonzero
finite-dimensional $k$-algebra.

Since $k$ is an algebraic closure of $\kappa(\mathfrak p)$, the
natural composite
\[
\mpk\longrightarrow\msZ\longrightarrow
\msZ_{\mathfrak p}\longrightarrow
\kappa(\mathfrak p)\hookrightarrow k
\]
makes $k$ an extension field of $\mpk$.
 For clarity, in this proof we distinguish the original
$\mpk$-algebras from their scalar extensions to $k$ and write
\[
\msH_k=k\otimes_{\mpk}\msH,
\qquad
\msP_k=k\otimes_{\mpk}\msP.
\]
 These are the algebras denoted simply by $\msH$ and $\msP$
in the preceding principal-series discussion. We also continue to
write $e$ and $X_i$ for $1\otimes e\in\msH_k$ and
$1\otimes X_i\in\msP_k$, respectively.

 Consider the composite algebra homomorphism
\[
\psi_{\mathfrak p}:
\msH
\longrightarrow
\msH_{\mathfrak p}
\longrightarrow
A
\longrightarrow
A/\mathfrak mA,
\]
where the first map is localization and the other two maps are the
canonical quotient maps.
The homomorphism $\psi_{\mathfrak p}$ induces a $k$-algebra
homomorphism
$
\Theta:
\msH_k\longrightarrow\overline A,$
such that $\Theta(\lambda\otimes h)
=
\lambda\otimes\psi_{\mathfrak p}(h).$
 We regard $\overline A$ as a right $\msH_k$-module via
$\Theta$. Since $\overline A\neq0$ is finite-dimensional over $k$,
it has a simple right $\msH_k$-module quotient $L$.

 Since the image of $e$ in $A$ is zero, $e$ acts trivially on
$\overline A$ and hence also on $L$. Thus
\begin{equation}\label{Le zero}
Le=0.
\end{equation}
  Since $k$ is
algebraically closed,   there exist
$0\neq w\in L$
and
$a=(a_1,\ldots,a_r)\in(k^\times)^r$
such that
\begin{equation}\label{wX_i}
 wX_i=a_iw\  (1\leq i\leq r).
\end{equation}
We define a
$k$-algebra homomorphism
$
\chi_{a,k}:\msP_k\rightarrow k$,
such that
$\chi_{a,k}(X_i)= a_i
$ for $1\leq i\leq r$.
 Restricting $\chi_{a,k}$ to
$\msP=\mpk[X_1^{\pm1},\ldots,X_r^{\pm1}]$
 gives an $\mpk$-algebra homomorphism
\[
\chi_a:\msP\longrightarrow k,
\qquad
X_i\longmapsto a_i.
\]
Put
$
\mathfrak q=\ker\chi_a.$
Then
$
\mathfrak q\in\Spec\msP.$

We claim that
$
\mathfrak q\cap\msZ=\mathfrak p.$
Let
$
\rho:
\msZ_{\mathfrak p}
\rightarrow
\kappa(\mathfrak p)
=
\msZ_{\mathfrak p}/\mathfrak m$
be the natural quotient map. For $z\in\msZ$, write
$
\bar z=\rho(z/1)\in\kappa(\mathfrak p)\subseteq k.$
On $A/\mathfrak mA$, and hence on $\overline A$ and $L$, the element
$z$ acts as multiplication by $\bar z$. Therefore
by \eqref{wX_i} we have
$$\chi_a(z)w=wz=\bar zw.$$
It follows that
$
\chi_a(z)=\bar z.$
Thus $\mathfrak p\subseteq \mathfrak q\cap\msZ$.
Furthermore if $z\in\mathfrak q\cap\msZ$, then $\bar z=\chi_a(z)=0$. It follows that
$
z/1\in\mathfrak m=\mathfrak p\msZ_{\mathfrak p}.$
Hence there exists
$
s\in\msZ\setminus\mathfrak p$
such that
$
sz\in\mathfrak p.$
Since $\mathfrak p$ is prime and $s\notin\mathfrak p$, we obtain
$z\in\mathfrak p$. Therefore we have $
\mathfrak q\cap\msZ=\mathfrak p.$

Since
$
\mathfrak q=\ker\chi_a,$
the homomorphism $\chi_a$ induces an injection
$\msP/\mathfrak q\hookrightarrow k,$
and therefore an embedding
$
\kappa(\mathfrak q) =\Frac(\msP/\mathfrak q)\hookrightarrow k.$ By  Lemma~\ref{height one q link},
there exists at most one unordered pair of distinct indices
$\{i,j\}$ such that
$X_i=q^{\pm1}X_j$ in
$\kappa(\mathfrak q)
=\Frac(\msP/\mathfrak q)$.
Consequently, there exists at most one unordered pair of distinct
indices $\{i,j\}$ such that $a_i$ and $a_j$ are $q$-linked.

Let $k_a$ be the one-dimensional right $\msP_k$-module
afforded by the character $\chi_{a,k}$, and set
$I_k(a)
=
k_a\otimes_{\msP_k}\msH_k.$
Since $L$ is simple as a right $\msH_k$-module, by \eqref{wX_i}
there is a surjective homomorphism of right $\msH_k$-modules
\[
\vi:I_k(a)\twoheadrightarrow L
\]
such that $\vi(1\otimes 1)= w $.
Thus $L$ is a simple composition factor of $I_k(a)$.
By Lemma~\ref{principal series main}  we have
$
Le\neq0.$
This contradicts \eqref{Le zero}. Therefore
$
\msH_{\mathfrak p}e\msH_{\mathfrak p}
=
\msH_{\mathfrak p}.$
\end{proof}

\section{The canonical approximation and the double-centralizer property}
In this section we complete the proof of the double-centralizer property. The argument is divided into three parts. We first construct a canonical left $\add(\msT)$-approximation of $\msH$. The height-one localization result proved in Proposition \ref{prop:height-one-fullness} is then used to show that this approximation splits after localization at every height-one prime of the Bernstein center. Finally, we prove that the cokernel of this approximation embeds into a finite direct sum of copies of $\msT$. This allows us to apply the abstract criterion established in \S 4.

\subsection{Construction and height-one splitting}
We first construct the approximation required in the double-centralizer criterion. Theorem \ref{injection} and the finite generation of $\msT$ over the center give the required embedding into a finite direct sum of copies of $\msT$.
\begin{lemma}\label{first left approx}
There exists an injective homomorphism of right $\msH$-modules
\[
\ttf:\msH\longrightarrow \msT^{\oplus N}
\]
which is a left $\add(\msT)$-approximation.
\end{lemma}

\begin{proof}
By Lemma \ref{T finite free over Z} we may
choose $\msZ$-generators
$
t_1,\ldots,t_N$
of $\msT$ and define
\begin{equation*}\label{f def}
\ttf:\msH\longrightarrow \msT^{\oplus N},
\qquad
\ttf(h)=(t_1h,\ldots,t_Nh)
\end{equation*}
for $h\in\msH$.
By Theorem \ref{injection} we conclude that $\ttf$ is
injective.
Since the $\msZ$-generators $t_1,\ldots,t_N$ also generate $\msT$
as an $\msS$-module, by Lemma~\ref{left approx} we conclude that $\ttf$ is
a left $\add(\msT)$-approximation.
\end{proof}

We next study the behavior of this approximation after localization. The key point is that the height-one localization property $\msH_{\mathfrak p}e\msH_{\mathfrak p}=\msH_{\mathfrak p}$ implies that
$\msH_{\mathfrak p}\in\add(\msT_{\mathfrak p})$. This makes the localized approximation split.
Let
$\mathfrak p\in\Spec\msZ$. By localization the map $\ttf$  induces a right $\msH_{\mathfrak p}$-module homomorphism
$\ttf_{\mathfrak p}:
\msH_{\mathfrak p}\longrightarrow
\msT_{\mathfrak p}^{\oplus N}$.

\begin{lemma}\label{height one splitting}
Let
$
\mathfrak p\in\Spec\msZ,
\
\operatorname{ht}\mathfrak p=1.
$
Then we have
$\msH_{\mathfrak p}\in\add(\msT_{\mathfrak p}),$
and   $\ttf_{\mathfrak p}:
\msH_{\mathfrak p}\longrightarrow
\msT_{\mathfrak p}^{\oplus N}$
is a left $\add(\msT_{\mathfrak p})$-approximation. Consequently,
$\ttf_{\mathfrak p}$ is a split monomorphism.
\end{lemma}

\begin{proof}
By Proposition~\ref{prop:height-one-fullness} we have
$\msH_{\mathfrak p}e\msH_{\mathfrak p}=
\msH_{\mathfrak p}.$
Hence there exist
$a_1,\ldots,a_m,b_1,\ldots,b_m\in\msH_{\mathfrak p}$
such that
$1=\sum_{i=1}^m a_i e b_i.$
Define
\[
\gamma:
\msH_{\mathfrak p}
\longrightarrow
(e\msH_{\mathfrak p})^{\oplus m},
\qquad
\gamma(h)=(eb_1h,\ldots,eb_mh),
\]
for $h\in\msH_{\mathfrak p}$,
and
\[
\ga':
(e\msH_{\mathfrak p})^{\oplus m}
\longrightarrow
\msH_{\mathfrak p},
\qquad
\ga'(x_1,\ldots,x_m)=\sum_{i=1}^m a_ix_i,
\]
for $(x_1,\ldots,x_m)\in(e\msH_{\mathfrak p})^{\oplus m}$.
Since $1=\sum_{i=1}^m a_i e b_i$ we have
$
 \ga'\circ\gamma
=
id_{\msH_{\frak p}}.$
Thus
$
\msH_{\mathfrak p}\in\add(e\msH_{\mathfrak p}).
$
Since $n\geq 2$, we have $\delta\in\afLanr$, where $\delta$ is defined in \eqref{the element e}. Hence by \eqref{tensor-decomp}, $e\msH=e_\delta\msH$ is a direct summand of $\msT$. It follows that
$
e\msH_{\mathfrak p}\in\add(\msT_{\mathfrak p}),
$
and hence
\begin{equation}\label{Hp add Tp}
\msH_{\mathfrak p}\in\add(\msT_{\mathfrak p}).
\end{equation}

We now show that $\ttf_{\mathfrak p}$ is a left
$\add(\msT_{\mathfrak p})$-approximation. Let
$
\varphi:
\msH_{\mathfrak p}\longrightarrow\msT_{\mathfrak p}
$
be an $\msH_{\mathfrak p}$-homomorphism. Write
\[
\varphi(1)= t/s ,
\qquad
t\in\msT,\quad
s\in\msZ\setminus\mathfrak p.
\]
Define
\[
\varphi_t:\msH\longrightarrow\msT,
\qquad
\varphi_t(h)=th
\]
for $h\in\msH$.
Since $\ttf$ is a left $\add(\msT)$-approximation, there exists a right $\msH$-module homomorphism
\[
\ttg:\msT^{\oplus N}\longrightarrow\msT
\]
such that
$
\varphi_t=\ttg\circ \ttf.$
By localization, $\ttg$ induces
$
\ttg_{\mathfrak p}:
\msT_{\mathfrak p}^{\oplus N}\longrightarrow \msT_{\mathfrak p}$ and $\vi_t$ induces
$(\varphi_t)_{\mathfrak p}:\msH_{\mathfrak p}\longrightarrow\msT_{\mathfrak p}$.
Since $
\varphi_t=\ttg\circ \ttf$ we have
$
(\varphi_t)_{\mathfrak p}
=
\ttg_{\mathfrak p}\circ\ttf_{\mathfrak p}.$
Furthermore, since $s\in\msZ\setminus\mathfrak p$, $s/1$ is invertible in $\msZ_{\mathfrak p}$. Let
\[
\widetilde \ttg=(1/s) \ttg_{\mathfrak p}:\msT_{\mathfrak p}^{\oplus N}\longrightarrow \msT_{\mathfrak p}.
\]
Then
$
\varphi=\widetilde \ttg\circ \ttf_{\mathfrak p}.
$
Hence $\ttf_{\mathfrak p}$ is a left
$\add(\msT_{\mathfrak p})$-approximation.
Consequently,  by \eqref{Hp add Tp}, there exists
$
\tth_{\mathfrak p}:
\msT_{\mathfrak p}^{\oplus N}
\rightarrow
\msH_{\mathfrak p}$
such that
$
\tth_{\mathfrak p}\circ \ttf_{\mathfrak p}
=
\operatorname{id}_{\msH_{\mathfrak p}}.
$
Therefore $\ttf_{\mathfrak p}$ is a split monomorphism.
\end{proof}
The splitting of the localized approximation will be used to
control the cokernel of $\ttf$ at height-one primes. We introduce
this cokernel in the next subsection.
\subsection{Torsion-freeness of the cokernel}

Let
\begin{equation*}
\msC=\Coker(\ttf).
\end{equation*}
Then we have the following short exact sequence
\begin{equation}\label{main ses}
0\longrightarrow \msH
\xrightarrow{\,\ttf\,}
\msT^{\oplus N}
\xrightarrow{\,\pi\,}
\msC
\longrightarrow0
\end{equation}
where $\pi:\msT^{\oplus N}\rightarrow\msC$ is the canonical projection.
We first record the local consequence of the splitting result above.
\begin{lemma}\label{Cp projective}
Let
$
\mathfrak p\in\Spec\msZ,
\
\operatorname{ht}\mathfrak p=1.$
Then $\msC_{\mathfrak p}$ is a projective
$\msZ_{\mathfrak p}$-module. In particular,
$\msC_{\mathfrak p}$ is a torsion free   $\msZ_{\mathfrak p}$-module.
\end{lemma}

\begin{proof}
By \eqref{main ses} we have the following
exact sequence
\[
0\longrightarrow\msH_{\mathfrak p}
\xrightarrow{\,\ttf_{\mathfrak p}\,}
\msT_{\mathfrak p}^{\oplus N}
\xrightarrow{\,\pi_{\mathfrak p}\,}\msC_{\mathfrak p}
\longrightarrow0.
\]
Hence by Lemma~\ref{height one splitting},  we have
\begin{equation}\label{split localized}
\msT_{\mathfrak p}^{\oplus N}
\cong
\msH_{\mathfrak p}\oplus\msC_{\mathfrak p}.
\end{equation}
By Lemma \ref{T finite free over Z},  there exist
an integer $m\geq1$ such that
$
\msT
\cong
\msZ^{\oplus m}.$
It follows that
$
\msT_{\mathfrak p}
\cong
\msZ_{\mathfrak p}^{\oplus m}.
$
Thus $\msT_{\mathfrak p}$  is a projective
$\msZ_{\mathfrak p}$-module. Consequently,
by \eqref{split localized},
$\msC_{\mathfrak p}$ is a projective $\msZ_{\mathfrak p}$-module.
Note that a projective module over
a domain is torsion-free.
Since $\msZ_{\mathfrak p}$ is a domain,
we conclude that
$\msC_{\mathfrak p}$
 is a torsion free $\msZ_{\mathfrak p}$-module.
\end{proof}

We now use the preceding height-one property to prove that
$\msC$ itself is torsion-free. To this end, we need the following elementary commutative algebra
observation, which allows us to choose a torsion element with a prime
annihilator.

\begin{lemma}\label{torsion prime ann}
Let $R$ be a Noetherian domain and let $M$ be an $R$-module.
Suppose that $M$ has nonzero torsion. Then there exists a nonzero
torsion element $m\in M$ such that
$
\mathfrak p:=\operatorname{Ann}_R(m)$
is a nonzero prime ideal of $R$.
\end{lemma}

\begin{proof}
Let
$
T
=
\{x\in M\mid rx=0
\text{ for some }0\neq r\in R\}$
be the torsion submodule of $M$. By assumption, $T\neq0$.
Consider
$
\mathcal A
=
\{\operatorname{Ann}_R(x)\mid 0\neq x\in T\}.$
Since $R$ is Noetherian,  $\mathcal A$ has a
maximal element with respect to inclusion. Choose
$
0\neq m\in T$
such that
$
\mathfrak p:=\operatorname{Ann}_R(m)
$
is maximal in $\mathcal A$.
Since $m$ is a torsion element,
$\mathfrak p\neq(0).$

It remains to prove that $\mathfrak p$ is prime. Suppose that
$
ab\in\mathfrak p$ and $a\notin\mathfrak p.$
Then
$
am\neq0.$
Since $m\in T$, we have
$ am\in T.$
Moreover,
$\mathfrak p
\subseteq
\operatorname{Ann}_R(am).$
By the maximality of $\mathfrak p$ in $\mathcal A$, it follows that
$
\operatorname{Ann}_R(am)=\mathfrak p.
$
Since
$
ab\in\mathfrak p=\operatorname{Ann}_R(m),$
we have
$
b\in\operatorname{Ann}_R(am)=\mathfrak p.$
Hence $\mathfrak p$ is prime.
\end{proof}

The following proposition will be used to pass from torsion-freeness at height-one localizations to global torsion-freeness.

\begin{proposition}\label{height one torsion free}
Let $R$ be a Noetherian unique factorization domain. Assume that
there is a short exact sequence of right $R$-modules
\[
0\longrightarrow F
\xrightarrow{\,f\,}
E
\xrightarrow{\,\pi\,}
M
\longrightarrow0,
\]
where $F$ is a projective $R$-module and $E$ is a torsion-free $R$-module.
Suppose further that $M_{\mathfrak p}$ is a torsion-free
$R_{\mathfrak p}$-module for every height one prime ideal
$\mathfrak p$ of $R$.
Then $M$ is a torsion-free $R$-module.
\end{proposition}

\begin{proof}
Suppose, to the contrary, that $M$ has nonzero torsion. Then by Lemma \ref{torsion prime ann} there exists a nonzero torsion element
$c\in M$ such that
$\mathfrak r:=\operatorname{Ann}_R(c)$
is a nonzero prime ideal of $R$.

Suppose first that $\operatorname{ht}\mathfrak r=1.$
Then $c/1\neq0$ in $M_{\mathfrak r}.$
Since $\mathfrak r\neq(0)$, there exist a non-zero element $z$ in $\mathfrak r=\operatorname{Ann}_{R}(c).$
As $R$ is a domain and $z\in\operatorname{Ann}_{R}(c)$, we have $z/1\neq0$ in $R_{\mathfrak r}$ and $(c/1)(z/1)=0.$
Thus $M_{\mathfrak r}$ contains a nonzero torsion element,
contrary to the hypothesis. Therefore
$\operatorname{ht}\mathfrak r\neq1.$
Since $R$ is a domain and
$(0)\subsetneq\mathfrak r,$
it follows that
$\operatorname{ht}\mathfrak r\geq2.$
Hence there exists a prime ideal $\mathfrak p_1$ such that
\[
(0)\subsetneq\mathfrak p_1\subsetneq\mathfrak r.
\]
Since $R$ is a unique factorization domain and $\mathfrak p_1\not=0$,
$\mathfrak p_1$ contains an irreducible element $x$.
Since every irreducible element in a unique factorization domain
is prime, $(x)$ is a prime ideal.
Hence $R/(x)$ is a domain.
Choose
$y\in\mathfrak r\setminus\mathfrak p_1.$
Since $(x)\subseteq\mathfrak p_1,$
we have $y\notin(x)$. Therefore
$\overline y:=y+(x)\neq0$ in $R/(x).$

Since $F$ is a projective $R$-module, it is a direct summand of a free $R$-module $L$. Hence $F/Fx$ is a direct summand of the free $R/(x)$-module $L/Lx$. Since $R/(x)$ is a domain, $L/Lx$ is a torsion-free $R/(x)$-module. Therefore the direct summand $F/Fx$ of  $L/Lx$ is also a torsion-free $R/(x)$-module.

Choose $u\in E$ such that
$\pi(u)=c.$
Since
$x,y\in\mathfrak r=\operatorname{Ann}_R(c),$
we have
\[
\pi(ux)=cx=0,
\qquad
\pi(uy)=cy=0.
\]
As
$\ker(\pi)=\operatorname{Im}(f),$
there exist $a,b\in F$ such that
$
f(a)=ux,$
$f(b)=uy.$
Since $f$ is $R$-linear,
\[
f(ay)=f(a)y=uxy=uyx=f(b)x=f(bx).
\]
As $f$ is injective,
$ay=bx.$ It follows that
$\overline a\,\overline y=\overline b\,\overline x=0$
in $F/Fx$.
Since
$\overline y\neq0$
in the domain $R/(x)$, and $F/Fx$ is torsion-free
$R/(x)$-module, we have
$\overline a=0.$
Hence $
a\in Fx,$
so there exists $d\in F$ such that
$a=dx.$
Consequently,
$
ux
=
f(a)
=
f(dx)
=
f(d)x,
$
and therefore
$\bigl(u-f(d)\bigr)x=0$ in $E$.
Since $E$ is torsion-free $R$-module and $x\neq0$,
we conclude that $u=f(d).$
It follows that
$c=\pi(u)=\pi(f(d))=0,$
contradicting the choice of $c$.
Therefore $M$ is torsion-free $R$-module.
\end{proof}

We now apply Proposition~\ref{height one torsion free} to the short exact sequence
\eqref{main ses}.

\begin{proposition}\label{C torsionfree}
The $\msZ$-module $\msC$ is torsion-free.
\end{proposition}

\begin{proof}
By Lemma~\ref{H is finite rank}, $\msH$ is a
projective $\msZ$-module. By Lemma~\ref{T finite free over Z},
$\msT^{\oplus N}$ is a free $\msZ$-module, and hence is
torsion-free since $\msZ$ is a domain. Moreover, by
Lemma~\ref{Cp projective}, $\msC_{\mathfrak p}$ is a torsion-free $\msZ_{\mathfrak p}$-module for every height-one
prime ideal $\mathfrak p$ of $\msZ$. Finally, by
Proposition~\ref{finite generated} and Lemma~\ref{Z is UFD}, $\msZ$ is a Noetherian unique factorization
domain. The assertion now follows from
Proposition~\ref{height one torsion free} applied to \eqref{main ses}.
\end{proof}

\subsection{Embedding of the cokernel}
We now use Proposition \ref{C torsionfree} to construct the
embedding required in the double-centralizer criterion.
\begin{lemma}\label{C embed H}
There exists an integer $d\geq0$ and an injective homomorphism of
right $\msH$-modules
$\beta:\msC\rightarrow \msH^{\oplus d}.$
Consequently, there is an injective homomorphism of right
$\msH$-modules
$
\theta:\msC\rightarrow \msT^{\oplus Nd}.$
\end{lemma}

\begin{proof}
By Lemma \ref{T finite free over Z}
the module $\msT^{\oplus N}$ is a finitely generated $\msZ$-module. It follows that  $\msC$ is a finitely generated $\msZ$-module, since
it is a quotient of $\msT^{\oplus N}$. Hence, by Proposition~\ref{C torsionfree} and
\cite[Lem.~4.33(ii)]{Rotman}, there exist an integer $d\geq0$ and
an injective $\msZ$-module homomorphism
\[
\ell:\msC\longrightarrow\msZ^{\oplus d}.
\]
For each $1\leq i\leq d$, let
$
\operatorname{pr}_i:\msZ^{\oplus d}\longrightarrow \msZ
$
be the $i$th coordinate projection, and define
\[
\ell_i:=\operatorname{pr}_i\circ\ell:
\msC\longrightarrow \msZ.
\]
For each $1\leq i\leq d$, define
\[
\widetilde\ell_i:
\msC\longrightarrow\Hom_{\msZ}(\msH,\msZ)
\]
by
$
\widetilde\ell_i(c)(h)=\ell_i(ch)
\
(c\in\msC,\ h\in\msH).
$
Then each $\widetilde\ell_i$ is a homomorphism of right
$\msH$-modules, where the right $\msH$-module structure on
$\Hom_{\msZ}(\msH,\msZ)$ is given by
\begin{equation*}
(F\cdot a)(h)=F(ah)
\qquad
(F\in\Hom_{\msZ}(\msH,\msZ),\ a,h\in\msH).
\end{equation*}
Define
\[
\widetilde\ell:
\msC\longrightarrow
\Hom_{\msZ}(\msH,\msZ)^{\oplus d},
\qquad
c\longmapsto
\bigl(
\widetilde\ell_1(c),\ldots,\widetilde\ell_d(c)
\bigr).
\]
If $\widetilde\ell(c)=0$, then
$\ell_i(c)=\widetilde\ell_i(c)(1)=0$
for $1\leq i\leq d$.
Hence $\ell(c)=0$. Since $\ell$ is injective, $c=0$. Therefore
$\widetilde\ell$ is injective.
By Brown--Gordon--Stroppel \cite[Th.~5.2 and (2.1)]{BGS},
$\Hom_{\msZ}(\msH,\msZ)$ is isomorphic to $\msH$
as a right $\msH$-modules. Therefore there is an injective homomorphism of right $\msH$-modules
\[
\beta:\msC\longrightarrow\msH^{\oplus d}.
\]
Finally,  the injective homomorphism
$
\ttf:\msH\rightarrow\msT^{\oplus N}
$
induces an injective homomorphism
\[
\ttf^{\oplus d}:
\msH^{\oplus d}\longrightarrow\msT^{\oplus Nd}.
\]
Therefore
\[
\theta:=\ttf^{\oplus d}\circ\beta:
\msC\longrightarrow\msT^{\oplus Nd}
\]
is an injective homomorphism of right $\msH$-modules.
\end{proof}

\subsection{Proof of the main theorem}
By Lemma~\ref{first left approx} and Lemma~\ref{C embed H}, the hypotheses of the abstract double-centralizer criterion are satisfied. The following theorem, which confirms Conjecture 3.8.8 of \cite{DDF}, now follows from Proposition~\ref{embedding criterion}.

\begin{theorem}\label{main}
Let $\mpk$ be a field of characteristic $0$, let
$\vep\in\mpk^\times$ be not a root of unity, and let
$n\geq2$.
The algebra homomorphism $$\xr : \ \afHrk\lra\End_{\afSrk} \Bigl( \bigoplus_{\lambda \in \afLanr} x_\lambda \afHrk \Bigr)^\oop$$
defined in \eqref{xir} is an algebra isomorphism.
\end{theorem}

\begin{proof}
Recall the short exact sequence
$
0\rightarrow\msH
\xrightarrow{\,\ttf\,}
\msT^{\oplus N}
\xrightarrow{\,\pi\,}
\msC
\rightarrow0,
$
given in \eqref{main ses},
where $\ttf$ is a left $\add(\msT)$-approximation.
By Lemma~\ref{C embed H}, there exist an integer $d\geq0$ and
an injective homomorphism of right $\msH$-modules
\[
\theta:\msC\longrightarrow\msT^{\oplus Nd}.
\]
Define
\[
\vartheta=\theta\circ\pi:
\msT^{\oplus N}\longrightarrow\msT^{\oplus Nd}.
\]
Since $\theta$ is injective,
$
\ker \vartheta
=
\ker\pi
=
\ttf(\msH).$
Hence
\[
0\longrightarrow\msH
\xrightarrow{\,\ttf\,}
\msT^{\oplus N}
\xrightarrow{\,\vartheta\,}
\msT^{\oplus Nd}
\]
is exact.
By Proposition~\ref{finite generated}, $\msS$ is left Noetherian.  Now the assertion follows from
Proposition~\ref{embedding criterion}.
\end{proof}

\begin{remark}
There is  a degenerate analogue of Theorem~\ref{main}.
For $n\geq r$,  this analogue is straightforward to establish (see \cite[Th.~3.8]{BrundanIvanov}).
For $2\leq n<r$, injectivity of the natural algebra homomorphism follows
from the faithfulness of the right polynomial representation of the
degenerate affine Hecke algebra $\mathsf{AH}_r$. More precisely, the direct summand
corresponding to $(r,0,\ldots,0)$ is naturally isomorphic,
as a right $\mathsf{AH}_r$-module, to
$
\mpk[x_1,\ldots,x_r],$
where
\[
f\cdot x_j=fx_j,\qquad
f\cdot s_i=s_i(f)+\partial_i(f),
\qquad
\partial_i(f)=\frac{f-s_i(f)}{x_i-x_{i+1}}
\]
for $f\in \mpk[x_1,\ldots,x_r]$.
This right polynomial representation is faithful. The arguments developed in \S4--\S6 can then be applied to the degenerate setting to establish a degenerate analogue of Theorem~\ref{main}. We omit the proof here.
\end{remark}

Let $\mpk$ be a field of characteristic $0$, and let
$\vep\in\mpk^\times$ be not a root of unity.  Assume $n\geq2$.
Let $U(\widehat{\mathfrak{gl}}_n)_\mpk$ be the quantum affine $\frak{gl}_n$ over $\mpk$. By \cite[Th. 3.8.1]{DDF}, the natural algebra homomorphism
\begin{equation}\label{zr}
\zeta_r:U(\widehat{\mathfrak{gl}}_n)_\mpk\ra\afSrk
\end{equation}
is surjective.
Applying Theorem~\ref{main}, we obtain the following result, which describes the extended affine Hecke algebra as the full centralizer of the quantum affine
$\frak{gl}_n$-action on affine tensor space.
\begin{theorem} \label{main2}
The natural commuting actions of
$U(\widehat{\mathfrak{gl}}_n)_\mpk$ and $\afHrk$ on
$\Omega_{\mpk}^{\otimes r}$ satisfy
$$ \afHrk \cong \operatorname{End}_{U(\widehat{\mathfrak{gl}}_n)_\mpk} (\Omega_{\mpk}^{\otimes r})^{\operatorname{op}} . $$
In other words, the extended affine Hecke algebra is the full
centralizer of the quantum affine $\mathfrak{gl}_n$-action.
\end{theorem}

\section{ Consequences for the center of affine quantum Schur algebras}

The double-centralizer theorem established in the previous section
allows us to describe the center of the affine quantum Schur algebra in
terms of the extended affine Hecke algebra and the quantum affine
$\mathfrak{gl}_n$.

We continue to work under the standing assumptions that $\mpk$ is a field of characteristic $0$, $\vep\in\mpk^\times$ is not a root of unity, and $n\geq2$. Throughout this section, $Z(A)$ denotes the center of an algebra $A$.

\subsection{The correspondence of centers}
The double-centralizer property established in Theorem~\ref{main}
yields a natural isomorphism between the centers of the extended affine
Hecke algebra and the affine quantum Schur algebra.
\begin{lemma}\label{center correspondence}
The restriction of the algebra homomorphism $\xi_r$ defined in \eqref{xir} to
$Z(\afHrk)$ induces an algebra isomorphism
$$
Z(\afHrk)
\xrightarrow{\ \sim\ }
Z(\afSrk),
$$
where we use the canonical identification
$
Z(\afSrk^{\operatorname{op}})=Z(\afSrk).$
\end{lemma}
\begin{proof}
Clearly we have
$ \xi_r(Z(\afHrk))\subseteq \operatorname{End}_{\afHrk} (\Omega_{\mpk}^{\otimes r})^{\operatorname{op}}\cap\operatorname{End}_{\afSrk} (\Omega_{\mpk}^{\otimes r})^{\operatorname{op}}=Z(\afSrk^{\operatorname{op}}). $
On the other hand, let
$ c\in Z(\afSrk^{\operatorname{op}}). $ Then we have $c\in\operatorname{End}_{\afSrk} (\Omega_{\mpk}^{\otimes r})^{\operatorname{op}}$.
Hence by Theorem \ref{main} there exists a unique
$h\in\afHrk$ such that
$  c=\xi_r(h). $ Since $c\in\afSrk^{\operatorname{op}}=\operatorname{End}_{\afHrk} (\Omega_{\mpk}^{\otimes r})^{\operatorname{op}}$ we have
 $ \xi_r(hx-xh) = c\xi_r(x)-\xi_r(x)c=0. $
for every $x\in\afHrk$. Therefore, since  $\xi_r$ is injective, we have
$ hx=xh $ for every $x\in\afHrk$. Consequently,
$ c=\xi_r(h)\in\xi_r(Z(\afHrk)). $
The assertion follows.
\end{proof}

\subsection{ Explicit description of the center}
We now compare the above description of the center with the central elements of the quantum affine $\frak{gl}_n$.
For $m\geq 1$, let $\sfz_m^\pm$ denote the central element of  $U(\widehat{\mathfrak{gl}}_n)_\mpk$
defined in \cite[(2.2.1.2)]{DDF}. We write
$
Z(n)
=
\mpk[\sfz_m^+,\sfz_m^-\mid m\geq 1]
$
for the commutative subalgebra generated by  $\sfz_m^\pm$ ($m\geq 1$) of $U(\widehat{\mathfrak{gl}}_n)_\mpk$. Let
$
Z(n,r)
=
\zeta_r(Z(n))$
where $\zeta_r$
is defined in \eqref{zr}.

The following theorem describes the relation between the center of the affine quantum Schur algebra and that of the quantum affine $\mathfrak{gl}_n$. It confirms the conjecture formulated in \cite[5.2.4]{DDF}.

\begin{theorem}\label{center Z0}
We have
$\zeta_r(Z(U(\widehat{\mathfrak{gl}}_n)_\mpk))=
Z(n,r)=
Z\bigl(\afSrk\bigr).$
\end{theorem}

\begin{proof}
Since $\zr$ is surjective and $Z(n)\subseteq Z(U(\widehat{\mathfrak{gl}}_n)_\mpk)$, we have
\begin{equation}\label{Z(n,r) subset}
 Z(n,r)\subseteq \zeta_r(Z(U(\widehat{\mathfrak{gl}}_n)_\mpk))\subseteq Z\bigl(\afSrk\bigr).
\end{equation}
By the Bernstein description of the center of the extended affine Hecke algebra,
$$
Z(\afHrk)
=
\mpk[X_1^{\pm1},\ldots,X_r^{\pm1}]^{\mathfrak S_r}=
\mpk[X_1,\ldots,X_r]^{\mathfrak S_r}
\mpk[X_1^{-1},\ldots,X_r^{-1}]^{\mathfrak S_r}.
$$
Let
\[
p_m^+=\sum_{i=1}^rX_i^m,\qquad
p_m^-=\sum_{i=1}^rX_i^{-m}
\]
for $m\geq1$. Since $\mpk$ has characteristic $0$,  we have
$
\mpk[X_1,\ldots,X_r]^{\mathfrak S_r}
=
\mpk[p_1^+,\ldots,p_r^+]$
and
$
\mpk[X_1^{-1},\ldots,X_r^{-1}]^{\mathfrak S_r}
=
\mpk[p_1^-,\ldots,p_r^-].$
Consequently,
\[
Z(\afHrk)
=
\mpk[p_1^+,\ldots,p_r^+,p_1^-,\ldots,p_r^-].
\]
By \cite[(3.5.5.2)]{DDF} we have
$
\zeta_r(\sfz_m^\pm)=\xi_r(p_m^\pm)
$
for $m\geq 1$.
Therefore, by Lemma~\ref{center correspondence},  and under the
canonical identification
$Z(\afSrk^{\text{op}}
)=Z(\afSrk)$ we have
$$
Z(\afSrk)=\xi_r(Z(\afHrk))
=
Z(n,r).$$
Consequently, by \eqref{Z(n,r) subset} we have $\zeta_r(Z(U(\widehat{\mathfrak{gl}}_n)_\mpk))=
Z(n,r)=
Z\bigl(\afSrk\bigr).$
\end{proof}

\end{document}